\documentclass[reqno]{amsart}
\usepackage{amsmath}
\usepackage{amsfonts}
\usepackage{amssymb}
\usepackage{amstext}
\usepackage{amsbsy}
\usepackage{amsopn}
\usepackage{amsthm}
\usepackage{amsxtra}
\usepackage{upref}

\newcommand{\x}{\ensuremath{\underline{x}}}
\newtheorem{thm}{Theorem}[section]

\newtheorem{lem}[thm]{Lemma}
\newtheorem{prop}[thm]{Proposition}
\theoremstyle{definition}

\theoremstyle{remark}

\numberwithin{equation}{section}

\newtheorem*{example*}{Example}

\def\del{\partial}

\def\var{\textrm{var}}

\def\wt{\widetilde}

\def\text#1{\textrm{#1}}

\def\emptyset{\varnothing}

\def\e{\varepsilon}

\def\vf{\varphi}
\def\a{\alpha}
\def\b{\beta}

\def\d{\delta}

\def\D{\mathbb D}
\def\l{\lambda}
\def\L{\Lambda}
\def\W{\mathcal W}
\def\s{\sigma}

\def\x{\times}
\def \R{\mathbb R}
\def \N{{\mathbb N}}
\def\E{\mathbb E}
\def \Z{\mathbb Z}

\def\ov{\overline}
\def\un{\underline}

\def\om{\omega}

\def\C{\mathbb C}

\def\D{\mathbb D}

\def\({\biggl(}
\def\){\biggr)}

\def\<{\bold\langle}
\def\>{\bold\rangle}

\DeclareMathOperator{\dist}{dist}

\DeclareMathOperator{\const}{const}

\DeclareMathOperator{\Hor}{Hor}
\DeclareMathOperator{\loc}{loc}
\DeclareMathOperator{\diam}{diam}

\DeclareMathOperator{\supp}{supp}
\DeclareMathOperator{\I}{I}
\DeclareMathOperator{\II}{II}
\DeclareMathOperator{\modi}{new}

\title[Generic Points for Horocycle Flows]{The Generic Points for the Horocycle Flow on a Class of Hyperbolic Surfaces with Infinite Genus}
\author{Omri Sarig}
\author{Barbara Schapira}
\thanks{O.S. was partially supported by an NSF grant DMS-0652966 and by an Alfred P. Sloan Research Fellowship. B.S. was  partially supported by the A.N.R. project T.E.M.I.
(2005-2008).
}
\address{O. Sarig, Mathematics Department, The Pennsylvania State University, University Park, PA 16802 USA}
\email{sarig@math.psu.edu}
\address{B. Schapira,
L.A.M.F.A. UMR 6140,
Facult\'e de Math\'ematiques et Informatique,
Universit\'e Picardie Jules Verne,
33 rue St Leu
80000 Amiens, France}
\email{Barbara.Schapira@u-picardie.fr}
\date{March 9, 2008}
\begin{document}
\begin{abstract}
A point is called generic for a flow preserving  an infinite ergodic invariant Radon measure, if its orbit satisfies the conclusion of the ratio ergodic theorem  for every pair of continuous functions with compact support and non-zero integrals.
The generic points for horocycle flows on hyperbolic surfaces of finite genus are understood, but there are no results in infinite genus. We give such a result, by characterizing the generic points for $\Z^d$--covers.
\end{abstract}
\maketitle

\section{Introduction}

\subsection*{Generic Points}
Suppose $\phi^t:X\to X$ is a continuous flow on a second countable locally compact metric space $X$. Let $C_c(X)$ denote the set of all continuous functions with compact support.
A point $x\in X$ is called {\em generic} for an invariant  Radon measure $m$, if
\begin{enumerate}
\item $m(X)<\infty$, and for all $f\in C_c(X)$,
$
\frac{1}{T}{\int_0^T f(\phi^t x)dt}\xrightarrow[T\to\infty]{}\frac{1}{m(X)}\int f dm
$,
\item $m(X)=\infty$, and for all $f,g\in C_c(X)$ with non-zero integrals,
$$
\frac{\int_0^T f(\phi^t x)dt}{\int_0^T g(\phi^t x)dt}\xrightarrow[T\to\infty]{}\frac{\int f dm}{\int g dm}.
$$
\end{enumerate}
If $m$ is  ergodic and conservative, then $m$--almost every $x$ is generic, because of the ratio ergodic theorem.
The question is to identify this set of full measure.
\subsection*{Horocycle Flows}
Suppose $M$ is a connected hyperbolic surface, and let $T^1 M$ denote its unit tangent bundle (made of all tangent vectors of length one). The (stable) {\em horocycle} of $\om\in T^1 M$ is the set $\textrm{Hor}(\om)$ of all unit tangent vectors $\om'$ such that $d(g^s\om',g^s\om)\xrightarrow[s\to\infty]{}0$, where $g^s:T^1 M\to T^1 M$ is the geodesic flow, and $d$ is the hyperbolic metric of $T^1 M$.
This is a smooth curve.
The (stable) {\em horocycle flow} is the flow $h^t:T^1 M\to T^1 M$  which moves a vector $\om\in T^1 M$ along $\textrm{Hor}(\om)$ at unit speed,  in the positive direction (to determine the orientation, lift to the universal cover $T^1\D$, $\D:=\{z\in\C:|z|<1\}$).

If $M$ has finite volume, then the generic points for the horocycle flow are understood thanks to the works of Furstenberg \cite{F} (compact surfaces), and Dani \& Smillie \cite{DS} (surfaces of finite area). Burger \cite{Bu}  characterized the symmetrically generic points for a large class of hyperbolic surfaces of finite genus and infinite volume, where `symmetrically generic' means that  $\int_0^T$ is replaced by  $\int_{-T}^{T}$. Schapira \cite{Scha} characterized the symmetrically generic points for all hyperbolic surfaces of finite genus.

These results can be summarized as saying that all non-periodic, non-wandering horocyles are generic for  the volume measure in case $vol(T^1M)<\infty$, or symmetrically generic for a certain singular infinite Radon measure which is carried by the non-wandering set of the horocycle flow in case $vol(T^1 M)=\infty$  (see \cite{Bu}, \cite{Scha} for details).

Ratner \cite{Ra} generalized the part of these results pertaining to {\em finite} invariant measures to all unipotent flows, whence to all horocycle flows on all hyperbolic surfaces. One of her results is that the horocycle flow on a hyperbolic surface of infinite volume (e.g. a surface with infinite genus) does not admit finite invariant measures except possibly trivial measures on periodic horocycles encircling cusps.
But there could be   globally supported {\em infinite} invariant Radon measures (Babillot \& Ledrappier \cite{BL}). The problem of describing the generic points for these measures is completely open.

The purpose of this paper is to solve this problem for  the simplest possible class of hyperbolic surfaces of infinite genus: $\Z^d$--covers of compact hyperbolic surfaces.

\subsection*{$\Z^d$--Covers} All surfaces in this paper are assumed to be connected.

A  hyperbolic surface $M$ is called a {\em regular $\Z^d$--cover of a compact hyperbolic surface $M_0$} (or just `$\Z^d$--cover'), if there is an onto map $p:M\to M_0$ such that (1) every $x\in M_0$ has a neighborhood $V_x$ such that every connected component of $p^{-1}(V_x)$ is mapped isometrically by $p$ onto $V_x$, (2) the group
$$
\textrm{Deck}(M,p):=\{D:M\to M: D\textrm{ is an isometry s.t. }p\circ D=D\}
$$
is isomorphic to $\Z^d$ (elements of $\textrm{Deck}(M,p)$ are called  {\em deck transformations}), and (3) for every $x\in M$, there is $\wt{x}\in M$ such that $p^{-1}(x)=\{D(\wt{x}):D\in\textrm{Deck}(M,p)\}$.

The deck transformations act on $T^1 M$ by their differentials. We abuse notation and use the same notation for $D$ and its differential.

Choose some connected fundamental domain $\wt{M}_0$ for the action of the group of deck transformations on $T^1 M$. Then every $\om\in T^1 M$ can be associated with a unique $\xi(\om)\in\Z^d$ such that
$
\om\in D_{\xi(\om)}[\wt{M}_0]$.
We call $\xi(\om)$ the {\em $\Z^d$--coordinate of $\om$}. It is useful to think of $T^1 M$ as of a $\Z^d$--array of copies of $\wt{M}_0$, tagged by their $\Z^d$--coordinates.

The {\em asymptotic cycle of a vector} $\om\in T^1 M$ is the following limit, if it exists:
$$
\Xi(\om):=\lim\limits_{T\to\infty}\frac{1}{T}\xi_T(\om), \textrm{ where } \xi_T(\om):=\xi(g^T\om)
$$
 (compare with \cite{Schw}). Let
$$
{\mathfrak C}:=\ov{\textrm{conv}}\left(\{\Xi(\om):\om\in T^1 M\textrm{ s.t. $\Xi(\om)$ exists}\}\right)\subset\R^d,
$$
where $\ov{\textrm{conv}}$ denote the closure of the convex hull.

\subsection*{Babillot--Ledrappier Measures}
The ergodic invariant Radon measures for horocycle flows on $\Z^d$--covers are known. To list them, fix a parametrization $\textrm{Deck}(M,p)=\{D_\xi:\xi\in\Z^d\}$ such that $D_\xi\circ D_\eta=D_{\xi+\eta}$.  Then
\begin{enumerate}
\item  For every homomorphism $\vf:\Z^d\to\R$, there exists
a unique (infinite) horocycle ergodic invariant Radon measure $m_\vf$ such that
$
m_\vf\circ D_{\xi}=e^{\vf(\xi)}m_\vf$ for all $\xi\in\Z^d$, and $m_\vf(\wt{M}_0)=1$ (\cite{BL}, Theorem 1.2). We call these measures the {\em Babillot--Ledrappier measures}.
\item There is a vector $\Xi_\vf\in \R^d$ such that $\Xi(\om)=\Xi_\vf$ $m_\vf$--almost everywhere;  the vector $\Xi_\vf$ determines $\vf$; and the set of all possible $\Xi_\vf$ is equal to the interior of $\mathfrak C$, $int(\mathfrak C)$ (\cite{BL}, Corollary 6.1 and \cite{BL1}, Proposition 1.1).
\item  Every ergodic invariant Radon measure for the horocycle flow on a $\Z^d$--cover of a compact hyperbolic surface is a constant times a Babillot--Ledrappier measure \cite{Sa}.
\end{enumerate}
The (normalized) volume measure  corresponds to $\vf\equiv 0$, and its almost sure asymptotic cycle is $\Xi_0=0$.

\subsection*{Main Result} Suppose $M$ is a connected regular $\Z^d$--cover of a compact hyperbolic surface $M_0$. The purpose of the paper is to prove:
\begin{thm}\label{mainthm}
A vector $\om\in T^1 M$ is generic for some horocycle ergodic invariant Radon measure $m$ iff $\Xi(\om)\in int(\mathfrak C)$. In this case $\Xi(\om)=\Xi_\vf$ and $m=cm_\vf$  for some uniquely determined homomorphism $\vf:\Z^d\to\R$ and $c>0$.
\end{thm}
\noindent
The hyperbolicity of the geodesic flow makes it easy to construct  geodesics without asymptotic cycles, thus there are many non-wandering horocycles which are not equidistributed with respect to any measure.  This should be contrasted with the finite genus case discussed above.

It is interesting that the mere existence of an asymptotic cycle is not enough for genericity, but that the value is important as well.

\subsection*{The Proof}
The proof that every $\om$ with asymptotic cycle in $int(\mathfrak C)$ is generic uses harmonic analysis.  It is a combination of methods from \cite{LS} for approximating the Birkhoff integral of a function by a symbolic dynamical quantity, and techniques from \cite{BL1} (building on  \cite{L}) for finding the asymptotic behavior of this quantity. A subtle difference is that the arguments of \cite{LS} only work for almost every $\om$, whereas we need an argument which uses explicit assumptions on $\om$ (see the proof of lemma \ref{basiclemma} for details).

The main contribution of the paper is the  converse statement: every generic vector has asymptotic cycle in $int(\mathfrak C)$. We do not know how to do this using  harmonic analysis, because the harmonic analytic tools we have fail for  $\om\in T^1M$ s.t. $\xi_T(\om)/T\to\del\mathfrak C$. We use a different method based on a certain built-in approximate exchangeability structure for the Lebesgue measure on horocycles.

\subsection*{Notational Convention} $a=b\pm c$ means $|a-b|<c$ or $\|a-b\|<c$ depending on the context. $a=e^{\pm c}b$ means $e^{-c}\leq \frac{a}{b}\leq e^c$. Our error bounds $c$ are always very generous, and rarely optimal.

\section{Preparations: Symbolic Dynamics}\label{symbsection}
\subsection*{Generalities}
A {\em subshift of finite type} with set of {\em states} $S$ and {\em
  transition matrix}
$A=(t_{ij})_{S\x S}$ $(t_{ij}\in\{0,1\})$ is  the set
$$
\Sigma:=\{x=(x_i)\in S^\Z:\forall i\in\Z,\,t_{x_i x_{i+1}}=1\}
$$
together with the action of the {\em left shift map} $\s:\Sigma\to\Sigma$, $\s({x})_k=x_{k+1}$, and the metric $d({x},{y})=\sum_{k\in\Z}\frac{1}{2^{|k|}}(1-\d_{x_k y_k})$.
There is a {\em one--sided} version $\s:\Sigma^+\to\Sigma^+$
 obtained by replacing $\Z$ by $\N\cup\{0\}$.

A {\em cylinder} in $\Sigma$ is a set of the form $[a_{-m},\ldots,\dot{a}_0,\ldots,a_{n}]:=\{x\in\Sigma:
x_{-m}^n=a_{-m}^n\}$, where the notation $x_{-m}^n$ means $(x_{-m},\ldots,x_n)$, and  the dot designates the zeroth coordinate. Cylinders in $\Sigma^+$ are defined similarly (with $m=0$). A word $\un{a}$ is called {\em admissible} if $[\un{a}]\neq\emptyset$. The cylinders form a basis of clopen sets for the topology of $\Sigma$ (or $\Sigma^+$). The length of a word $\un{a}$ is denoted by $|\un{a}|$.

A subshift of finite type is {\em topologically mixing} iff $\exists m$ s.t. all the entries of $A^m$ are positive.  Any  topologically mixing subshift of finite type has a number $M_{br}\in\N$
and a collection of admissible words of length $M_{br}$ $\mathfrak B:=\{\un{w}_{ab}:a,b\in S\}$ such that
$$
[a,\un{w}_{ab},b]\neq \emptyset.
$$
We fix such a collection, and refer to its elements as {\em `Markovian bridges'}.

Suppose $F$ is a real valued function on $\Sigma$ or $\Sigma^+$.
The  Birkhoff sums of functions $F$ are denoted by $F_n$:
$$
F_n:=F+F\circ\s+\cdots+F\circ\s^{n-1}.
$$
A function $F$ is said {\em to depend only on non-negative coordinates} if $x_0^\infty=y_0^\infty$ implies $F(x)=F(y)$.
The {\em variation} of such a function is
$$
\var(F):=\sum_{n=1}^\infty \sup\{|F(x)-F(y)|:x_0^{n-1}=y_0^{n-1}\}.
$$
If $F$ is H\"older continuous, then this number is finite.

Let $G$ be a group, assumed for simplicity to be abelian.
The {\em skew-product} over $T:X\to X$ with the {\em cocycle} $f:X\to G$ is the map $T_f:X\x G\to X\x G$, $T(x,\xi):=\bigl(T(x),\xi+f(x)\bigr)$.

The {\em suspension semi-flow} over  $T:X\to X$ and {\em height function} $r^\ast:X\to\R^+$ is the semi-flow $\vf^s:X_{r^\ast}\to X_{r^\ast}$, where
$$
X_{r^\ast}:=\{(x,t): x\in X, 0\leq t< r^\ast(x)\}
$$
and $\vf^s(x,t):=\bigl(T^n x, t+s-r^\ast_n(x)\bigr)$ where $n$ is chosen s.t. $0\leq t+s-r^\ast_n(x)<r^\ast(T^n x)$.
If $T$ is invertible, then this semi-flow has a unique extension to a flow.
 The suspension (semi)-flow can be identified with the (semi)-flow
 $(x,t)\mapsto (x,t+s)$ on $(X\x\R)/\sim$
 (respectively $X\x\R^+/\sim$) where $\sim$ is the orbit relation
of the skew-product $T_{-r^\ast}$. Both descriptions shall be used below.

\subsection*{Symbolic Dynamics for the Geodesic Flow} Let $p:M\to M_0$ be a $\Z^d$--cover of a compact connected orientable hyperbolic surface $M_0$.
We describe the geodesic flow on $g^s:T^1M\to T^1M$ as a suspension flow, whose base is a skew--product,
 whose base is a subshift of finite type.
This description is well--known \cite{PS},\cite{Po},\cite{BL1}. It can be obtained
by a lifting argument from the Bowen--Series  symbolic
 dynamics of $g^s:T^1 M_0\to T^1 M_0$ given by \cite{BS}, \cite{Se1},\cite{Se2}, as in.

\begin{lem}\label{bowen}
Fix $i:T^1 M_0\to T^1 M$ 1-1  with image $\wt{M}_0$ and
 s.t. $dp\circ i=id$.
There exist a topologically mixing two--sided subshift of finite type $(\Sigma, T)$, a H\"older continuous function $r:\Sigma\to\R$
which depends only on the non-negative coordinates, a function $f:\Sigma\to\Z^d$ s.t. $f(x)=f(x_0,x_1)$, a H\"older function $h:\Sigma\to\R$,  and a
H\"older continuous map $\pi:\Sigma\x\Z^d\x\R\to T^1 M$ with the following properties:
\begin{enumerate}
\item $r^\ast:=r+h-h\circ\s$ is non-negative, and there exists a constant $n_0$ such that $\inf r^\ast_{n_0}>0$, where $r^\ast_{n_0}:=\sum_{k=0}^{n_0-1}r^\ast\circ\s^k$.
\item $\pi:(\Sigma\x\{0\})_{r^\ast}\to \wt{M}_0$ is a surjective two-to-one map, where
$$
(\Sigma\x\{0\})_{r^\ast}=\{(x,0,t):0\leq t<r^\ast(x)\}.
$$
$\pi$ is one-to-one everywhere except on a set of  which maps into a countable union of geodesics in $T^1 M$.
\item If $Q_{\xi_0,t_0}(x,\xi,t)=(x,\xi+\xi_0,t+t_0)$, then $\pi\circ Q_{(\xi_0,t_0)}=[{g}^{t_0}\circ D_{\xi_0}]\circ\pi$ for all $(\xi_0,t_0)\in\Z^d\x\R$;
\item $\pi\circ T_{(f,-r^\ast)}=\pi$ where $T_{(f,-r^\ast)}(x,\xi,t)=\bigl(\s x,\xi+f(x),t-r^\ast(x)\bigr)$;
\item Suppose $\om=\pi(x,\xi,t), \om'=\pi(x',\xi',t')$. Then
$$
\hspace{0.8cm}
\exists p,q\geq 0\textrm{ s.t }
\begin{cases}
x_p^\infty=(x')_q^\infty &\\
t-t'=h(x)-h(x')+r_p(x)-r_q(x') &\\
\xi-\xi'=f_q(x')-f_p(x)
\end{cases}
\Rightarrow \exists \tau\textrm{ s.t. }\om'=h^\tau(\om).
$$
\item For every $\om$, there are at most countably many points on
  $\{h^t\om\}_{t\in\R}$
with more than one $\pi$--preimage.
\item Nonarithmeticity \cite{Sh}: $
\ov{\<(-r_n(x),f_n(x)):T^n x=x,\ n\in\N\>}=\R\x\Z^d
$. (See also \cite{C}.)
\end{enumerate}
\end{lem}
\noindent
Most of lemma \ref{bowen} extends variable curvature, except for
 parts (2) and it immediate corollary  (6),  which we  only know how to prove using the
the explicit nature  of the Bowen--Series coding in constant negative curvature.

\subsection*{Symbolic Coordinates} The {\em symbolic coordinates} of $\om\in T^1 M$ are defined to be $(x,\xi,t)$ such that $\om=\pi(x,\xi,t)$ and
$
0\leq t<r^\ast(x)$.

We call $x=x(\om)$ the {\em $\Sigma$--coordinate}, $\xi=\xi(\om)$ the {\em $\Z^d$--coordinate}, and $t=t(\om)$ the {\em $\R$--coordinate} (of $\om$).

Some vectors have two sets of  symbolic coordinates, but the collection of such vectors intersects every horocycle at a negligible, even countable, set. Therefore such vectors are of no consequence to us.

\subsection*{The Symbolic Description of the Babillot-Ledrappier Measures}
Suppose a flow $\phi^t:X\to X$ has a Poincar\'e section $K$ with section map $T_K:K\to K$. If we represent the flow as a suspension over the section, then any $\phi$--invariant measure can be identified with the product of a $T_K$--invariant measure $\mu_K$ and the Lebesgue measure on $\R$, restricted to the suspension space. Thus any $\phi$--invariant measure is determined by the measure it induces on a Poincar\'e section.

The horocycle flow has a Poincar\'e section which can be naturally coded as $\Sigma^+\x\Z^d\x\R$. We describe the Babillot--Ledrappier measures in terms of the measures they induce on this section.

Let $S$ be the set of states of $\Sigma$, and fix some $P:S\to S$ such that $(P(a),a)$ is admissible. Define $\rho':\Sigma^+\x\Z^d\x\R\to \Sigma\x\Z^d\x\R$ by $\rho'(x^+,\xi,s)=(x,\xi,s+h(x))$ where $x_0^\infty=(x^+)_0^\infty$, and $x_{k}=P(x_{k+1})$ for $k<0$. By \cite{BM} (see also \cite{Sa})
\begin{enumerate}
\item $\rho:=\pi\circ\rho'$ maps $\Sigma^+\x\Z^d\x\R$ onto a Poincar\'e section for the horocycle flow on $T^1 M$;
\item if $\exists p,q$ s.t. $(x^+)_p^\infty=(y^+)_q^\infty$, $\xi-\xi'=f_q(x^+)-f_p(y^+)$, $s-s'=r_p(x)-r_q(x')$, then $\rho(x^+,\xi,s)$ and $\rho(y^+,\xi',s')$ lie on the same horocycle.
\end{enumerate}
There is a slight inconvenience in that  $\rho$ is countably-to-one (because of lemma \ref{bowen} part (4)). We shall deal with this problem by working with restrictions of $\rho$ to sets where it is one-to-one.

By (1), any measure $m$ on $T^1 M$ can be identified (locally, on subsets of the suspension space on which $\rho$ is one-to-one) with $\mu\x dt$, where  $\mu$ is some measure on $\Sigma^+\x\Z^d\x\R$.  By (2), if $\mu$ is invariant for the equivalence relation
\begin{equation}\label{equiv}
(x^+,\xi,s)\sim (y^+,\xi',s')\Longleftrightarrow \left\{ \begin{array}{rcl}
(x^+)_p^\infty&=&(y^+)_q^\infty\\
\xi-\xi'&=&f_q(x^+)-f_p(y^+)\\
s-s'&=&r_p(x)-r_q(x'),
\end{array}\right.
\end{equation}
then $\mu$ is invariant for the Poincar\'e map of the horocycle flow\footnote{ Here we use the general fact that a  measure is invariant for a transformation if and only if it is invariant for the corresponding orbit  relation. (A Borel measure $m$ is {\em invariant} for a Borel equivalence relation $\sim$, if $m\circ\kappa|_{\textrm{dom}(\kappa)}=m|_{\textrm{image}(\kappa)}$ for all partial Borel isomorphisms $\kappa:\textrm{dom}(\kappa)\to \textrm{image}(\kappa)$ such that $\kappa(x)\sim x$ for all $x\in \textrm{dom}(\kappa)$.)}, and therefore gives rise to a flow invariant measure on $T^1M$ (see \cite{BL} for details).

This was the way Babillot and Ledrappier constructed their family of infinite invariant measures for the horocycle flow: they constructed an infinite family of measures $\mu_\vf$ which are invariant under the equivalence relation (\ref{equiv}).

The construction uses  the thermodynamic formalism. We review some necessary facts (see e.g. \cite{PP}).
Let $\s:\Sigma^+\to\Sigma^+$ denote the {\em one}--sided version of $\Sigma$. 
The {\em topological pressure} of a continuous function $\varphi:\Sigma^+\to\R$ is the number $P_{top}(\varphi):=\sup\{h_\mu(\s)+\int\varphi d\mu\}$ where the supremum ranges over all shift invariant probability measures on $\Sigma^+$.
Define $P:\R^d\to\R$ implicitly by $u\mapsto P(u)$ where $P(u)=P$ is the root of
$$
P_{top}(-Pr+\<u,f\>)=1.
$$
This defines a $C^\infty$--diffeomorphism from $\R^d$ onto $int(\mathfrak C)$ \cite{BL}. 
This diffeomorphism can be defined canonically, and  does not depend on the particular choice of the coding, see \cite{BL}.

Any homomorphism $\vf:\Z^d\to\R$ can be identified with a vector $u_\vf\in\R^d$ by $\vf(\cdot)=\<u_\vf,\,\cdot\,\>$.
Define the operator $L_{-P(u_\vf),u_\vf}:C(\Sigma^+)\to C(\Sigma^+)$ by
$$
(L_{-P(u_\vf),u_\vf}F)(x)=\sum_{\s y=x}e^{-P(u_\vf) r(y)+\<u_\vf,f(y)\>}F(y).
$$
Ruelle's Perron-Frobenius Theorem says that there exists a probability measure $\nu_{\vf}$ and a positive
H\"older continuous function $\Psi_{\vf}$ on $\Sigma^+$ such that
$$
L_{-P(u_\vf),u_\vf}\Psi_\vf=\Psi_\vf\ ,\ L^{\ast}_{P(u_\vf),u_\vf}\nu_{\vf}=\nu_{\vf}\ ,\textrm{ and }\int\Psi_\vf d\nu_\vf=1.
$$
The Babillot--Ledrappier measure $m_\vf$ (on $T^1 M$) is given locally by $\mu_\vf\x dt$, where $\mu_\vf$ (a measure on the Poincar\'e section) is given in symbolic coordinates by
\begin{equation}\label{BLmeasures}
d\mu_\vf(x,\xi,s):=\const {e^{-P(u_\vf)s+\<u_\vf,\xi\>}d\nu_\vf d\xi ds}.
\end{equation}
This formula can be used to deduce  the following important consequences:
\begin{enumerate}
\item $m_\vf$ is quasi--invariant under the geodesic flow, and
$m_\vf\circ g^s=e^{-P(u_\vf)s}m_\vf$;
\item $\Xi_\vf=(\nabla P)(u_\vf)$ (corollary 6.1 in \cite{BL});
\item $u_\vf=-(\nabla H)(\Xi_\vf)$, where $H(\cdot)$ is minus the Legendre transform of $P(\cdot)$.
\end{enumerate}
 (The double minus sign in (3) is a convention from \cite{BL}.)

\subsection*{Symbolic Local Manifolds}
Suppose $\om$ has symbolic coordinates $(x,\xi,t)$ (i.e. $0\leq s+h(x)<r^\ast(x)$), and write $t=s+h(x)$.
The {\em symbolic local stable manifold}
of $\om=\pi(x,\xi,s+h(x))$ is defined by
$$
W^{ss}_{\loc}(\om):=\pi\{(y,\xi,s+h(y)):y_0^\infty=x_0^\infty\}.
$$
This is a subset of $\textrm{Hor}(\om)$, because of
lemma \ref{bowen} part (5).

If $W^{ss}_{\loc}(\om_1), W^{ss}_{\loc}(\om_2)$ intersect with positive measure, then they are equal up to sets of length zero. Indeed, for almost every $\om$ in the intersection  $\exists p$ s.t.
$
\om=\pi(x,\xi(\om_1),s(\om_1)+h(x))$, $x_0^\infty=x(\om_1)_0^\infty$ and $\om=\pi(\s^p x,\xi(\om_2),s(\om_2)+h(\s^p x))$ , $x_p^\infty=x(\om_2)_0^\infty$.
 This forces $x(\om_1)_p^\infty=x(\om_2)_0^\infty$,  $\xi(\om_2)=\xi(\om_1)+f_p(x(\om_1))$, $s(\om_2)=s(\om_1)-r_p(x(\om_1))$ (lemma \ref{bowen} part (1)). It follows from lemma \ref{bowen} parts (4) and (1) that $W^{ss}_{\loc}(\om_1)=W^{ss}_{\loc}(\om_2)$ up to sets of measure zero.

\begin{lem}\cite[prop. 4.5]{BL}\label{margulis}  Let $\ell_\om$ denote
 the hyperbolic length measure on the horocycle of $\om$.
If $\om=\pi(x,\xi,s+h(x))$, then
$\ell_\om[W^{ss}_{\loc}(\om)]=e^{-s}\psi(x_0,x_1,\ldots)$
where $\psi:\Sigma^+\to\R$ is H\"older continuous and positive. (In fact $\psi=\Psi_0$.)
\end{lem}

\section{Proof that a Generic Vector has Asymptotic Cycle in
$\textrm{\em int}(\mathfrak C)$}\label{section_converse}
In this section we prove
\begin{prop}\label{converseprop}
Let $M$ be a $\Z^d$--cover of a compact hyperbolic surface. If $\om_0$ is generic for some horocycle invariant Radon measure $m$, then there exists a homomorphism $\vf:\Z^d\to\R$ and $c>0$ such that $m=c m_\vf$ and  $\Xi(\om)=\Xi_\vf$.
\end{prop}

The proof is as follows.
Suppose $\om_0$ is generic for some horocycle ergodic invariant Radon measure $m$. By \cite{Sa}, there is a homomorphism $\vf:\Z^d\to\R$ and $c>0$ such that $m=c m_\vf$. Of course if $\om_0$ is generic for $cm_\vf$, then it is generic for $m_\vf$, therefore we may assume w.l.o.g. that $m=m_\vf$. We saw above that  $\Xi(\om)=\Xi_\vf$ $m_\vf$--a.e., for some $\Xi_\vf\in \textrm{int}(\mathfrak C)$. We aim at showing that $\Xi(\om_0)=\Xi_\vf$.

Recall that $\wt{M}_0\subset T^1 M$ is the set of vectors whose $\Z^d$--coordinate is equal to zero, and define the following objects:
\begin{eqnarray}
A_T&:=&\{h^t(\om_0):0\leq t\leq T\},\notag\\
\l_T&:=&\textrm{the normalized length measure on }A_T\cap\wt{M}_0\label{lambda}\\
&=&\frac{1}{\int_0^T 1_{\wt{M}_0}(h^t(\om_0))dt}\int_0^T
1_{\wt{M}_0}(h^t\om_0) \d_{h^t(\om_0)}dt\ \  (\d_x:=\textrm{Dirac at $x$}),\notag\\
\L_n(N,\e_0)&:=& \left\{\om\in\wt{M}_0:\om=\pi(x,\un{\xi},s)
\textrm{ and }\right.\left.\biggl\|\frac{f_N(\s^{nN}x)}{r_N^\ast(\s^{nN}x)}-\Xi_\vf\biggr\|<\e_0\right\}.
\label{Lambda}
\end{eqnarray}
Recall that $n_0$ is a constant such that $\min r^\ast_{n_0}>0$ (lemma \ref{bowen} part (1)). The key to the proof of proposition \ref{converseprop} is the following lemma:
\begin{lem}[Key Lemma]\label{keyineqlemma}
There are $C_0,T_0>0$  such that for every $\e_0>0$ there are constants $N(\e_0)>n_0$, $K(\e_0)>0$, and  $\tau(\e_0)$ such that for all  $T>\tau(\e_0)$
\begin{equation}\label{keyineq}
\sup_n \l_T\left\{\om:r_{nN(\e_0)}^\ast(x(\om))<\ln\tfrac{T}{T_0}-K(\e_0)\textrm{ and }\om\not\in \Lambda_n(N(\e_0),\e_0)
\right\}\leq
C_0\e_0.
\end{equation}
\end{lem}
The proof is somewhat technical, so we postpone it, and first explain   how to use the lemma to prove proposition \ref{converseprop}.

Fix $\e_0$ and set $N=N(\e_0)$, $K=K(\e_0)$, $\L_n:=\L_n(N(\e_0), \e_0)$, and $T^\ast:=\ln(T/T_0)$. Abbreviate
$
\{r^\ast_{nN}<T^\ast-K\}:=\{\om\in T^1M:r^\ast_{nN}(x(\om))<T^\ast-K \}.
$

Working in the symbolic model for $T^1 M$,
define the following functions, which we view as random variables on $A_T$:
 $A_k:=f_N\circ \s^{kN}\,1_{\{r_{kN}^\ast<T^\ast-K\}}$
and $B_k:=r^\ast_N\circ\s^{kN}1_{\{r_{kN}^\ast<T^\ast-K\}}$. Then for all $T$ large enough
$$
\hspace{-1.8cm}\E_{\l_T}\!\!\left(\frac{\sum_i A_i}{\sum_i B_i}\right)\!\!
=\E_{\l_T}\left(
\sum_i \frac{B_i}{\sum_j B_j}\left[\left(\frac{A_i}{B_i}-\Xi_\vf\right)1_{\{r^\ast_{iN}<T^\ast-K\}}\right]
\right)+\Xi_\vf\\
$$
\begin{eqnarray*}
&=&\!\!\!\!\Xi_\vf\pm
\sum_i \left\|\frac{B_i}{\sum_j B_j}\right\|_\infty\left\|\left(\frac{A_i}{B_i}-\Xi_\vf\right)1_{\{r^\ast_{iN}<T^\ast-K\}}\right\|_1\\
&=&\!\!\!\!\Xi_\vf\pm \sum_{i=1}^{[(T^\ast/\min r_{n_0}^\ast)+n_0]/N } \!\!\left(\frac{N\max r^\ast}{T^\ast-K- N\max r^\ast}\right)\left\|\left(\frac{A_i}{B_i}-\Xi_\vf\right)1_{\{r^\ast_{iN}<T^\ast-K\}}\right\|_1.
\end{eqnarray*}

We estimate the $L^1$-norm by breaking the support of $\l_T$  into $\wt{M}_0\cap\L_i$ and $\wt{M}_0\setminus\L_i$:
$$
\hspace{-8cm}\left\|\left(\frac{A_i}{B_i}-\Xi_\vf\right)1_{\{r^\ast_{iN}<T^\ast-K\}}\right\|_1
$$
\begin{eqnarray*}
&\leq&
\e_0{\l_T(\wt{M}_0\cap\Lambda_i)}+\left(\frac{N\max\|f\|_\infty}{\lfloor N/n_0\rfloor\min r^\ast_{n_0}}+\|\Xi_\vf\|\right)
\l_T(\wt{M}_0\cap\{r_{iN}^\ast<T^\ast-K\}\setminus\Lambda_i)\\
&\leq &\e_0+\left(\frac{N\max\|f\|_\infty}{\lfloor N/n_0\rfloor\min r^\ast_{n_0}}+\|\Xi_\vf\|\right)\sup_i \l_T(\wt{M}_0\cap \{r_{iN}^\ast<T^\ast-K\}\setminus \Lambda_i).
\end{eqnarray*}

Thus, the key lemma implies that for all $T$ large enough
\begin{equation}\label{limit}
\left\|\E_{\l_T}\left(\frac{\sum_i A_i}{\sum_i B_i}\right)-\Xi_\vf\right\|\leq \const\e_0
\end{equation}
with the constant independent of $T$ and $\e_0$.

But the quotient whose expectation we are calculating is nearly constant!
For every $\om\in\{h^t(\om_0):0<t<T\}$, $\sum_i B_i=T^\ast\pm (K+N\max r^\ast)$, and
\begin{eqnarray*}
\sum_i A_i
&=&\xi(g^{T^\ast}\om)\pm \left(K+N\right)n_0\frac{\max\|f\|}{\min r^\ast_{n_0}}\\
&=&\xi(g^{T^\ast}\om_0)\pm \left[\left(K+N\right)n_0\frac{\max\|f\|}{\min r^\ast_{n_0}}+\|\xi(g^{T^\ast}\om)-\xi(g^{T^\ast}\om_0)\|\right].
\end{eqnarray*}
Now  $\dist(g^{T^\ast}(\om), g^{T^\ast}(\om_0))\leq e^{-T^\ast}T=T_0$, so $\|\xi(g^{T^\ast}\om)-\xi(g^{T^\ast}\om_0)\|$ is bounded by some constant independent of $T^\ast$. Thus we obtain $\sum_i A_i=\xi(g^{T^\ast}\om_0)\pm \const$, where the constant is independent of $T^\ast$.

We see that
$$
\E_{\l_T}\left(\frac{\sum_i A_i}{\sum_i B_i}\right)=
\frac{\xi(g^{T^\ast}\om_0)\pm \const}{T^\ast\pm \const}=
[1+o(1)] \frac{1}{T^\ast}\xi(g^{T^\ast}\om_0)+o(1).
$$
Comparing this with (\ref{limit}), we see that for  $T$ large,
$\|\frac{1}{T^\ast}\xi(g^{T^\ast}\om_0)-\Xi_\vf\|<\const\e_0$, with the constant independent of $\e_0$.
Since $\e_0$ was arbitrary, $\om_0$ has asymptotic cycle $\Xi_\vf$, and  proposition \ref{converseprop} is proved (assuming lemma
\ref{keyineqlemma}).

\medskip
The remainder of the section contains the proof of  lemma \ref{keyineqlemma}. The proof requires  the construction of certain self maps of horocycles. These are constructed and studied in the next subsections.

\subsection*{Preparations I: Distortion Estimates}
We find uniform bounds on the Radon-Nikodym derivative of certain maps between local symbolic stable manifolds, equi\-pped with the hyperbolic length measure.

Define an equivalence relation on $\Sigma$  via
$$
x\sim y\Leftrightarrow \exists p,q\ \textrm{ s.t. }x_p^\infty=y_q^\infty
$$
(the equivalence class of $x$ can be viewed as the weak stable ``manifold'' of $x$ for the action of the shift).
If $x,y$ are  non-eventually periodic, then  the following are independent of the choice of $p,q$:
\begin{equation}\label{cocycles}
\begin{array}{rcl}
R^+(x,y)&:=&r_{q}(y)-r_{p}(x),\\
R(x,y)&:=&r_{q}(y)-r_{p}(x)+h(y)-h(x)=\lim\limits_{n\to\infty}[r^\ast_{q+n}(y)-r^\ast_{p+n}(x)],\\
F(x,y)&:=&f_p(x)-f_q(y)=\lim\limits_{n\to\infty}[f_{p+n}(x)-f_{q+n}(y)].\\
\end{array}
\end{equation}
These are minus the functions appearing in lemma \ref{bowen}, part 5. Note that $R^+(x,y)$ and $F(x,y)$ only depend on $x_0^\infty, y_0^\infty$.

Fix some admissible sequence  $x_0^\infty$ and two words $\un{a}, \un{b}$ (not necessarily of the same length) such that $(\un{a},x_0),(\un{b},x_0)$ are admissible and such that $a_0=b_0$. Define
$$
\kappa^\ast:\{w\in\Sigma: w_{-|\un{a}|}^\infty=\un{a}x_0^\infty\}\to\{w\in\Sigma: w_{-|\un{b}|}^\infty=\un{b}x_0^\infty\}
$$
by $$
\kappa^\ast(w)=\begin{cases}
(w_{-\infty}^{-|\un{a}|-1},\un{b},\dot{x}_0^\infty) & w=(w_{-\infty}^{-|\un{a}|-1},\un{a},\dot{x}_0^\infty)\\
(w_{-\infty}^{-|\un{b}|-1},\un{a},\dot{x}_0^\infty) & w=(w_{-\infty}^{-|\un{b}|-1},\un{b},\dot{x}_0^\infty)\\
w &\textrm{otherwise},
\end{cases}
$$
where the zero coordinate is at the beginning of the dotted word.
This induces the following self--map of $W^{ss}(\pi(x,\xi,s+h(x))$:
$$
\kappa[\pi(w,\xi,s+h(w))]=\pi(\kappa^\ast(w),\xi, s+h(\kappa^\ast(w))).
$$
This map  is well--defined and one--to--one on $W^{ss}_{\loc}(\pi(x,\xi,s+h(x)))$ minus a countable set (Lemma \ref{bowen}, part 2).

\begin{lem}\label{basichol1}
$\kappa:W^{ss}_{\loc}(\pi(x,\xi,s+h(x)))\to W^{ss}_{\loc}(\pi(x,\xi,s+h(x)))$ is absolutely continuous with respect to the hyperbolic length measure $\ell$, and there is some constant $D'>1$, independent of $x_0^\infty$, $\un{a}$, and $\un{b}$ such that
$$
(D')^{-1}\exp[-|R^+(\un{a} x_0^\infty,\un{b} x_0^\infty)|]\leq \frac{d\ell\circ\kappa}{d\ell}\leq D'\exp[|R^+(\un{a} x_0^\infty,\un{b} x_0^\infty)|].
$$
\end{lem}
\begin{proof}
Split the domain of $\kappa$ into three parts: $A$, where it flips $\un{a}$ to $\un{b}$, $B$ where it flips $\un{b}$ to $\un{a}$, and $C$, where $\kappa=id$.

We estimate the derivative on $A$.
Fix $n$ and a word $(x'_{-n},\ldots,x'_{-1})$ such that $(x_{-|\un{a}|}',\ldots,x_{-1}')=\un{a}$. We study the distortion in the length of the horocycle piece
$
K_{x_{-n}',\ldots,x_{-1}'}:=\{
\pi(w,\xi,s+h(w)):w_{-n}^\infty=((x')_{-n}^{-1},x_0^\infty)\}$:
$$
\hspace{-3cm}\ell[K_{x_{-n}',\ldots,x_{-1}'}]=\ell[\{
\pi(w,\xi,s+h(w)):w_{-n}^\infty=((x')_{-n}^{-1},x_0^\infty)\}]
$$
\begin{eqnarray*}
&=&\ell[\{
\pi(\s^{-n}w,\xi-f_n(\s^{-n}w),s+h(w)+r_n^\ast(\s^{-n}w)):w_{-n}^\infty=((x')_{-n}^{-1},x_0^\infty)\}]\\
&=&\ell[\{
\pi(w',\xi,s+h(\s^n w')+r_n^\ast(w')):(w')_{0}^\infty=((x')_{-n}^{-1},x_0^\infty)\}]\ \ (\because \ell\circ D_{f_n}=\ell)\\
&=&\ell[\{
\pi(w',\xi,s+h(w')+r_n(w')):(w')_{0}^\infty=((x')_{-n}^{-1},x_0^\infty)\}]\\
&=&e^{-r_{n}(x'_{-n},\ldots x'_{-1},x_0^\infty)}\psi(x'_{-n},\ldots x'_{-1},x_0^\infty)\ \ \  \ (\because \ell\circ g^s=e^{-s}\ell).
\end{eqnarray*}
A similar calculation shows that
$$
(\ell\circ\kappa)[K_{x_{-n}',\ldots,x_{-1}'}]=e^{-r_{n+|\un{b}|-|\un{a}|}(x'_{-n},\ldots x'_{-|\un{a}|-1},\un{b},x_0^\infty)}\psi(x'_{-n},\ldots x'_{-|\un{a}|-1},\un{b},x_0^\infty).
$$
Dividing, we see that
\begin{eqnarray*}
\frac{(\ell\circ\kappa)[K_{x'_{-n},\ldots,x'_{-1}}]}{\ell[K_{x'_{-n},\ldots,x'_{-1}}]}&=&\!\!\!\!\!\frac{\exp[r_n(x'_{-n},\ldots x'_{-1},x_0^\infty)]}{\exp[r_{n+|\un{b}|-|\un{a}|}(x'_{-n},\ldots x'_{-|\un{a}|-1},\un{b},x_0^\infty)]}\\
&&\hspace{4cm}\x\frac{\psi(x'_{-n},\ldots x'_{-|\un{a}|-1},\un{b},x_0^\infty)}{\psi(x'_{-n},\ldots,x'_{-1}x_0^\infty)}\\
&\leq &\!\!\!\!\! \left(\frac{\max\psi}{\min\psi}\right)e^{|R^+(\un{a}x_0^\infty,\un{b}x_0^\infty)|+\var (r)}.
\end{eqnarray*}
Thus, for a global constant $D'$ independent of $x,\un{a}$, and $\un{b}$,
$$
\frac{(\ell\circ\kappa)[K_{x'_{-n},\ldots,x'_{-1}}]}{\ell[K_{x'_{-n},\ldots,x'_{-1}}]}\leq D'\exp[|R(\un{a}x_0^\infty,\un{b}x_0^\infty)|].
$$
Since this estimate holds for all $n>|\un{a}|$, and since $\{K_{x_{-n}',\ldots,x_{-1}'}\}$ generate the Borel sigma--algebra of $A$, we get that $\ell\circ\kappa|_A\ll\ell$,  and  $d\ell\circ\kappa/d\ell\leq D'\exp[|R(\un{a}x_0^\infty,\un{b}x_0^\infty)|]$.
The lower bound on the derivative is  obtained in exactly the same manner, as do  the estimates for the distortion of $\kappa$ on $B$.
\end{proof}

We turn to discuss a different class of maps between local symbolic manifolds.
Fix two admissible sequences  $x_0^\infty$, $y_0^\infty$ and define
$$
\vartheta^\ast:\{w\in\Sigma: w_0^\infty=x_0^\infty\}\to \{w\in\Sigma: w_0^\infty=y_0^\infty\}
$$
by $\vartheta^\ast(w)=(w_{-\infty}^{-1},\un{w}_{w_{-1}y_0},\dot{y}_0^\infty)$, where
$\un{w}_{w_{-1}y_0}$ is a bridge word out of the (fixed) collection $\mathfrak B$, and
where the zero coordinate is the first symbol of the dotted word.

This induces the map from $W^{ss}_{\loc}(\pi(x,\xi,s+h(x))$ into $W^{ss}_{\loc}(\pi(y,\xi,s+h(y))$
$$
\vartheta[\pi(w,\xi,s+h(w))]=\pi(\vartheta^\ast(w),\xi, s+h(\vartheta^\ast(w))).
$$
 As before,  $\vartheta$ is well--defined and one--to--one on $W^{ss}_{\loc}(\pi(x,\xi,s+h(x)))$ minus a countable set. But we can not  claim it is surjective (because of the insertion of a fixed bridge word).

\begin{lem}\label{basichol2}
$\vartheta:W^{ss}_{\loc}(\pi(x,\xi,s+h(x)))\to W^{ss}_{\loc}(\pi(y,\xi,s+h(y)))$ is absolutely continuous with respect to the hyperbolic length measure $\ell$, and there is some constant $D''>1$, independent of $x_0^\infty$ and $y_0^\infty$ such that
$
{D''}^{-1}\leq {d\ell\circ\vartheta}/{d\ell}\leq D''$.
\end{lem}
\begin{proof}
Fix $n$ and a word $(w_{-n},\ldots,w_{-1})$ such that $[w_{-1},x_0]\neq\emptyset$,  and set
$$
K_{w_{-n},\ldots,w_{-1}}:=\{
\pi(w',\xi,s+h(w')):{(w')}_{-n}^\infty=(w_{-n}^{-1},x_0^\infty)\}.
$$
The same calculation as in the previous lemma shows that
\begin{eqnarray*}
\frac{(\ell\circ\vartheta)[K_{w_{-n},\ldots,w_{-1}}]}{\ell[K_{w_{-n},\ldots,w_{-1}}]}&=&\!\!\!\!\frac{\exp[r_n(w_{-n},\ldots w_{-1},x_0^\infty)]}{\exp[r_{n+|\un{w}_{w_{-1}y_0}|}(w_{-n},\ldots w_{-1},\un{w}_{w_{-1}y_0},y_0^\infty)]}\\
&&\hspace{3.5cm}\x\frac{\psi(w_{-n},\ldots w_{-1},\un{w}_{w_{-1}y_0},y_0^\infty)}{\psi(w_{-n},\ldots,w_{-1}x_0^\infty)}\\
&\leq &\!\!\!\! \left(\frac{\max\psi}{\min\psi}\right)\x\exp[\var(r)]\x \max_{\un{w}\in\mathfrak B}\left[\exp\max_{x\in[\un{w}]}|r_{M_{br}}(x)|\right].
\end{eqnarray*}
As in the proof of the previous lemma, this implies that  for some constant $D''_1>1$, independent of $x$ and $y$,
 $d\ell\circ\vartheta/d\ell\leq D''_1$.
A uniform lower bound $D''_2$ can be obtained in exactly the same manner.
\end{proof}

\subsection*{Preparations II: Permuting Points on a Horocycle}
We use the transformations discussed above to  construct certain maps which `permute' points on a given horocycle.
These bijections are defined symbolically in terms of the coding $\pi(x,\xi,s)$. They take the form
\begin{eqnarray*}
\theta[\pi(x,\xi,s+h(x))]&=&\pi\bigl(\theta^\ast(x),\xi+F(x,\theta^\ast(x)),s+h(x)+R(x,\theta^\ast(x))\bigr)\\
&=&\pi\bigl(\theta^\ast(x),\xi+F(x,\theta^\ast(x)),s+h(\theta^\ast x)+R^+(x,\theta^\ast(x))\bigr)
\end{eqnarray*}
 where $\theta^\ast$ will be defined below,  and   $F(\cdot,\cdot), R(\cdot,\cdot), R^+(\cdot,\cdot)$, given by (\ref{cocycles}), are designed to ensure that $\theta[\pi(x,\xi,s)]$ stays on the horocycle of $\pi(x,\xi,s)$. Definitions of this form make sense everywhere on the horocycle where $\pi$ is injective, therefore everywhere except a countable set.

The maps $\theta^\ast(x)$ we need fall into two groups (precise definitions and explanation of notation follow):
\begin{enumerate}
\item maps $\theta^\ast=\kappa_{n,N}^\ast$ which exchange the positions of two $N$--blocks in $x$,
\item maps $\theta^\ast=\vartheta_{T^\ast,\om_1,\om_2}^\ast$ which exchange some suffix of $x$ by another suffix.
\end{enumerate}

\subsubsection*{Exchanging blocks: $\kappa_{n,N}$}
Fix $n, N\in \N$. Define $\kappa_{n,N}^\ast$ on $\Sigma$ by exchanging
the places of the $N$--blocks $(x)_0^{N-1}$, $(x)_{nN}^{nN+N-1}$,
and then inserting bridge words (see \S\ref{symbsection})  at the right places to ensure admissibility:
\begin{multline*}
\kappa_{n,N}^{\ast}(x):=(x_{-\infty}^{-1},w_{x_{-1}x_{nN}},\textrm{\fbox{$\dot{x}_{nN}^{nN+N-1}$}},
w_{x_{nN+N-1},x_{N}},\\
x_{N}^{nN-1},w_{x_{nN-1},x_0},\textrm{\fbox{$x_0^{N-1}$}},w_{x_{N-1},x_{(n+1)N}},x_{(n+1)N}^\infty).
\end{multline*}
The zero coordinate is at the first symbol of the dotted word, and the boxed blocks are those that were switched.
This gives rise to the following self--map of the horocycle $\Hor(\pi(x,s,h(x))$:
$$
\kappa_{n,N}\bigl(\pi(x,\un{\xi},s+h(x))\bigr):=\pi\bigl(\kappa_{n,N}^{\ast}x,\xi+F(x,\kappa_{n,N}^{\ast}x),
s+h(\kappa_{n,N}^{\ast}x)+R^+(x,\kappa_{n,N}^{\ast}x)\bigr).
$$

\begin{lem}\label{properties_of_kappa_nN} The map $\kappa_{n,N}$  is  absolutely continuous and injective on a subset of full $\ell$--measure of  ${\Hor}(\pi(x,\xi,s+h(x)))$, and there exist  constants $C'_r\ge 1$, $E'$, and $L'$ independent of $n,N$ and $x$ such that
\begin{enumerate}
\item $(C_r')^{-1}\leq d\ell\circ\kappa_{n,N}/d\ell\leq C_r'$,
\item the $\Sigma$--coordinate of $\kappa_{n,N}(\pi(x,\xi,t))$ is $\s^{l}(\kappa_{n,N}^\ast(x))$ for some $|l|\leq L'$,
\item $\dist(\om,\kappa_{n,N}(\om))\leq E'$ $(\dist=\dist_{T^1 M})$.
\end{enumerate}
\end{lem}
\begin{proof}
The map $\kappa^\ast_{n,N}$ can be thought of locally as the map which exchanges the word $\un{a}$ by $\un{b}$, where
\begin{eqnarray*}
\un{a}\!\!\!&=&\!\!\!(x_0,\ldots,x_{(n+1)N-1})\\
\un{b}\!\!\!&=&\!\!\!(w_{x_{-1}x_{Nn}},{x}_{Nn}^{(n+1)N-1},
w_{x_{N(n+1)-1},x_{N}},
x_{N}^{Nn-1},w_{x_{Nn-1},x_0},x_0^{N-1},
w_{x_{N-1},x_{(n+1)N}}).
\end{eqnarray*}
Note that $|\un{a}|=(n+1)N$, and $|\un{b}|=(n+1)N+4M_{br}$, where $M_{br}$ is the  length of the bridge words in $\mathfrak B$.

The distortion of such maps was found in Lemma \ref{basichol1}.
 A direct calculation using the boundedness and H\"older continuity of $r$ shows  that there exist a constant $E_R'$, which only depends on $\max r$, $M_{br}$ and $\var (r)$ (but not on $N,n$), such that
$|R^+(\un{a}x_{(n+1)N}^\infty,\un{b}x_{(n+1)N}^\infty)|<E_R'$.
Part (1) follows from lemma \ref{basichol1}.

Next we study the $\Sigma$--coordinate of $\kappa_{n,N}(\pi(x,\xi,s+h(x)))$. This must be  $\s^{l}(\kappa^\ast_{n,N}x)$ with $l$ such that either $l\geq 0$ and
$$
0\leq s+h(\kappa^\ast_{n,N}x)+R^+(x,\kappa^\ast_{n,N}x)-r^\ast_l(\s^l\kappa^\ast_{n,N}x)< r^\ast(\s^l \kappa^\ast_{n,N}x).
$$
or $l<0$ and
$
0\leq s+h(\kappa^\ast_{n,N}x)+R^+(x,\kappa^\ast_{n,N}x)+r^\ast_{-l}(\s^{l}\kappa^\ast_{n,N}x)< r^\ast(\s^l \kappa^\ast_{n,N}x)$.
We have seen that $|R^+(x,\kappa_{n,N}^\ast x)|\leq E_R'$. Since $h$ is  bounded and
$\inf r^\ast_{n_0}>0$, $l$ must be bounded by some constant $L'=L'(\max |h|,E_r',\min r^\ast_{n_0})$. Part (2) follows.

Finally we study the $\Z^d$--coordinate of $\om':=\kappa_{n,N}(\pi(x,\xi,s+h(x)))$. If  $l$ is as above, then the symbolic coordinates of $\om'$ are (assuming for simplicity that $l>0$)
$$
(\s^l(\kappa_{n,N}^\ast(x)),\xi+F(\un{a}x_{(n+1)N}^\infty,\un{b}x_{(n+1)N}^\infty)+f_l(\kappa_{n,N}^\ast(x)), \textrm{something}).
$$
It is routine to check that  $\|F(\un{a}x,\un{b}x)\|$ is uniformly bounded by some constant $E_F'$ which only depends on $\max\|f\|$ and $M_{br}$.
This means that the difference in $\Z^d$--coordinates is no more than $\|F(\un{a}x_{(n+1)N}^\infty,\un{b}x_{(n+1)N}^\infty)+f_l(\kappa_{n,N}^\ast(x))\|\leq E_F'+L'\max\|f\|.$

Thus $\om, \om'$ lie in two copies of $\wt{M}_0\subset T^1 M$ which differ by a deck transformation whose $\Z^d$--index is bounded.
Since the diameter of $\wt{M}_0$ is finite, this implies that  $\dist(\om,\om')$ is uniformly bounded, whence part (3).
\end{proof}

\subsubsection*{Changing suffices: $\vartheta_{T^\ast,\om_1,\om_2}$}
Fix $T^\ast>0$ and two vectors $\om_1$,$\om_2$ on the same horocycle $\Hor(\om)$. We define a map $\vartheta_{T^\ast,\om_1,\om_2}:g^{-T^\ast}[W^{ss}_{\loc}(g^{T^\ast}\om_1)]\to g^{-T^\ast}[W^{ss}_{\loc}(g^{T^\ast}\om_2)]$ as follows: write $\om_i^\ast:=g^{T^\ast}(\om_i)=\pi(x_i^\ast,\xi_i^\ast,s_i^\ast+h(x_i^\ast))$ $(i=1,2)$, and define
\begin{enumerate}
\item $\vartheta^\ast_{x_1^\ast,x_2^\ast}:\{z^\ast\in\Sigma: (z^\ast)_0^\infty=(x_1^\ast)_0^\infty\}\to \{x^\ast\in\Sigma: (x^\ast)_0^\infty=(x_2^\ast)_0^\infty\}$ by
$$
\vartheta^\ast_{x_1^\ast,x_2^\ast}(z)=\bigl(z_{-\infty}^{-1},w_{z_{-1},({x}_2^\ast)_0},(\dot{x}_2^\ast)_0^\infty\bigr);
$$
\item $\wt{\vartheta}_{\om_1,\om_2}: W^{ss}_{\loc}(g^{T^\ast}\om_1)\to W^{ss}_{\loc}(g^{T^\ast}\om_2)$ by
$$
\wt{\vartheta}_{\om_1,\om_2}(\pi(z,\xi_1^\ast,s_1^\ast+h(z)))=\pi(\vartheta^\ast_{x_1^\ast,x_2^\ast}z,\xi_2^\ast,s_2^\ast+h(\vartheta^\ast_{x_1^\ast,x_2^\ast}z));
$$
\item The map $\vartheta_{T^\ast,\om_1,\om_2}:g^{-T^\ast}[W^{ss}_{\loc}(g^{T^\ast}\om_1)]\to g^{-T^\ast}[W^{ss}_{\loc}(g^{T^\ast}\om_2)]$ is
$$
\vartheta_{T^\ast,\om_1,\om_2}:=g^{-T^\ast}\circ \wt{\vartheta}_{\om_1,\om_2}\circ g^{T^\ast}.
$$
\end{enumerate}
\begin{lem}\label{properties_of_vartheta}  Suppose $\dist(g^{T^\ast}\om_1,g^{T^\ast}\om_2)\leq T_0$. The map $\vartheta_{T^\ast,\om_1,\om_2}$  is  an absolutely continuous injective map from a subset of full $\ell$--measure of $g^{-T^\ast}[W^{ss}_{\loc}(g^{T^\ast}\om_1)]$ into
$g^{-T^\ast}[W^{ss}_{\loc}(g^{T^\ast}\om_2)]$, and there exist  constants $C''_r\ge 1$, $E''$ and $L''$ which only depend on $T_0$ such that
\begin{enumerate}
\item $(C_r'')^{-1}\leq d\ell\circ\vartheta_{T^\ast,\om_1,\om_2}/d\ell\leq C_r''$,
\item the $\Sigma$--coordinate of $\vartheta_{T^\ast,\om_1,\om_2}(\pi(x,\xi,T))$ is $\s^{-m}(\vartheta^\ast_{x_1^\ast,x_2^\ast}\s^n x)$ with  $n\geq 0$ such that $|r_n^\ast(x)-T^\ast|\leq 5(\max r^\ast+\max|h|)$ and $|m-n|\leq L''$,
\item $\dist(\om,\vartheta_{T^\ast,\om_1,\om_2}\om)\leq E''$ $(\dist=\dist_{T^1 M})$.
\end{enumerate}
\end{lem}
\begin{proof}
Lemma \ref{basichol2} implies that $\wt{\vartheta}_{\om_1,\om_2}$ is absolutely continuous and that
$$
\frac{1}{D''}\leq \frac{d\ell\circ\wt{\vartheta}_{\om_1,\om_2}}{d\ell}\leq D''
$$
with $D''$ independent of $\om_i$. Since the geodesic flow contracts the length of horocycles uniformly, part (1) follows with $C_r'':=D''$.

To compare the symbolic coordinates of $\om$ and $\vartheta_{T^\ast,\om_1,\om_2}(\om)$, write $\om=\pi(x,\xi,t)$ with $0\leq t<r^\ast(x)$.
Since $\om\in g^{-T^\ast}W^{ss}_{\loc}(\om_1^\ast)$, there is some $n$ s.t.
\begin{eqnarray*}
g^{T^\ast}(\om)&=&\pi(\s^n(x),\xi+f_n(x),t+T^\ast-r_n^\ast(x))
\end{eqnarray*}
where $\s^n(x)_0^\infty=(x_1^\ast)_0^\infty$,
$\xi+f_n(x)=\xi_1^\ast$,
$t+T^\ast-r_n^\ast(x)=s_1^\ast+h(\s^n(x))$.
Thus
\begin{eqnarray*}
\vartheta_{T^\ast,\om_1,\om_2}(\om)&=&
g^{-T^\ast}\pi\biggl(\vartheta^\ast_{x_1^\ast,x_2^\ast}(\s^n x),\xi_2^\ast, s_2^\ast+h\bigl(\vartheta^\ast_{x_1^\ast,x_2^\ast}(\s^n x)\bigr)\biggr)\\
&=&\pi\left(
\begin{array}{c}
\s^{-m}\vartheta^\ast_{x_1^\ast,x_2^\ast}(\s^n x)\\
\xi_2^\ast-f_m(\s^{-m}\vartheta^\ast_{x_1^\ast,x_2^\ast}(\s^n x))\\
 s_2^\ast+h(\vartheta^\ast_{x_1^\ast,x_2^\ast}(\s^n x))-T^\ast+r_m^\ast(\s^{-m}\vartheta^\ast_{x_1^\ast,x_2^\ast}(\s^n x))
\end{array}
\right)
\end{eqnarray*}
with $m$ s.t.
$$
0\leq  s_2^\ast+h(\vartheta^\ast_{x_1^\ast,x_2^\ast}(\s^n x))-T^\ast+r_m^\ast(\s^{-m}\vartheta^\ast_{x_1^\ast,x_2^\ast}(\s^n x))<r^\ast(\s^{-m}\vartheta^\ast_{x_1^\ast,x_2^\ast}(\s^n x)).
$$

 The obvious bounds $|s_2^\ast|\leq \max r^\ast+\max|h|$, $|r^\ast_m-r_m|\leq 2\max|h|$, and the double inequality above  imply that
$$
|r_m(\s^{-m}\vartheta^\ast_{x_1^\ast,x_2^\ast}(\s^n x))-T^\ast|\leq 5(\max r^\ast+\max|h|).
$$
The identity $t+T^\ast-r_n^\ast(x)=s_1^\ast+h(\s^n(x))$ means that
$$
|r_n^\ast(x)-T^\ast|, |r_n(x)-T^\ast|\leq 5(\max r^\ast+\max|h|).
$$
We see that $|r_n(x)-r_m(\s^{-m}\vartheta^\ast_{x_1^\ast,x_2^\ast}(\s^n x))|\leq 10(\max r^\ast+\max|h|)$.
By construction,
\begin{equation}\label{m-n}
(\s^{-m}\vartheta^\ast_{x_1^\ast,x_2^\ast}(\s^n x))_0^{\infty}=(x_{n-m+M_{br}},\ldots,x_{n-1},{w}_{x_{n-1}(x^\ast_2)_0},(x_2^\ast)_0^\infty)
\end{equation}
This allows us to estimate $|m-n|$:
\begin{eqnarray*}
10(\max r^\ast+\max|h|)&\geq&
|r_n(x)-r_m(\s^{-m}\vartheta^\ast_{x_1^\ast,x_2^\ast}(\s^n x))|\\
&\geq &\inf |r_{|n-m+M_{br}|}|-2M_{br}\max |r|-2\var(r)\\
&\geq &\inf r_{|n-m+M_{br}|}^\ast-2\max|h|-2M_{br}\max |r|-2\var(r)\\
&\geq & \left\lfloor\frac{|n-m|}{n_0}\right\rfloor\inf r_{n_0}^\ast-2\max|h|-3M_{br}\max |r|-2\var(r).
\end{eqnarray*}
Since $\inf r^\ast_{n_0}>0$, we see that
$|m-n|\leq L''$, with $L''=L''(r,r^\ast,M_{br},h,n_0)$ with  $L''$  independent of $\om_1, \om_2$ and $T^\ast$.

We can now estimate the difference between the $\Z^d$ and $\R$ coordinates of  $\om=\pi(x,\xi,t)$ and $\om':=\vartheta_{T^\ast,\om_1,\om_2}(\om)$ which we write as $\om'=\pi(x',\xi',t')$.

We saw above  that $\xi=\xi_1^\ast-f_n(x)$. Comparing this to the  symbolic coordinates of $\om'$ found above, and recalling that $f(x)=f(x_0,x_1)$,   we see that
\begin{eqnarray*}
\|\xi'-\xi\|&\leq & \|\xi_1^\ast-\xi_2^\ast\|+\|f_m(x')-f_n(x)\|\\
&\leq & \|\xi_1^\ast-\xi_2^\ast\|+(L''+M_{br})\max\|f\|\ (\because (\ref{m-n})).
\end{eqnarray*}
Since by assumption $\dist(\om_1^\ast,\om_2^\ast)\leq T_0$, $\|\xi_1^\ast-\xi_2^\ast\|$ is bounded by a constant which only depends on $T_0$ and the geometry of $\wt{M}_0$. This implies that $\|\xi'-\xi\|$ is bounded above by some constant which only depends on $T_0, L''$ and $\max\|f\|$.

The upper bound on the difference between $\Z^d$--coordinates implies a uniform upper bound $E''$ on the distance between $\om$ and $\om'$ in $T^1 M$, whence part (3).
\end{proof}

\subsection*{Proof of the Key  Lemma}
Define once and for all the following constants:
\begin{itemize}
\item $\ell_{\min}:=\inf\limits_{\om'\in T^1 M}\ell(W^{ss}_{\loc}(\om'))=\min\psi$,   $d_{\max}:=\sup\limits_{\om'\in T^1 M} \diam[W^{ss}_{\loc}(\om')]$ (with the diameter measured using the intrinsic horocycle metric);
\item $T_0:=100d_{\max}$;
\item $C_r:=\max\{C_r', C_r''\} $, where $C_r',C_r''$ are as in lemmas \ref{properties_of_kappa_nN} and \ref{properties_of_vartheta} (with $T_0$ as above);
\item $L:=\max\{L', L''\}$ where $L', L''$ are as in lemmas \ref{properties_of_kappa_nN} and \ref{properties_of_vartheta} (with $T_0$ as above);
\item $E:=\max\{E', E''\}$ where $E', E''$ are as in lemmas \ref{properties_of_kappa_nN} and \ref{properties_of_vartheta} (with $T_0$ as above);
\item $K^\ast:=100L(\max r^\ast+\max |h|+\max |r|+\max\|f\|+\var(r))$, where $r,r^\ast,f$ are as in lemma \ref{bowen}, $L$ as in lemma \ref{basichol1};
\item Recall that  $\wt{M}_0\subset T^1 M$ is a fundamental domain for the action of the
covering group  on $T^1 M$. Fix $C=C(\wt{M}_0,E)$ so large that
$$
\dist(\om,\wt{M}_0)<10E\Longrightarrow \om\in\bigcup_{\|a\|<C}D_{\un{a}}(\wt{M}_0)=:\wt{M}_0(C),
$$
where $\{D_\xi:\xi\in\Z^d\}$ are the deck transformations of the cover.
\end{itemize}
Finally, fix some arbitrarily small $\e_0>0$.

\medskip
\noindent
{\em Step 1.\/} Application of Egoroff's theorem, and choice of $N(\e_0)$.

\medskip
\noindent
Recall that
$\frac{1}{T}\xi(g^T\om)\to \Xi_\vf$ as $T\to \infty$ $m_\vf$--almost everywhere.
Symbolically, this means that
$$
\frac{f_n(x)}{r_n^\ast(x)}\xrightarrow[n\to\infty]{}\Xi_\vf\textrm{ for $m_\vf$--a.e. $\om=\pi(x,\xi,t)$.}
$$
The reason is that this fraction is asymptotic to $\frac{1}{T_n}[\xi(g^{T_n}\om)-\xi(\om)]$,
where $T_n=$ is the $n$-th hitting time to the Poincar\'e section.

We construct a large set $\L$ where the limit is nearly achieved in finite time.
There exists an $N=N(\e_0)$ s.t.
$
m_\vf\{\om\in \wt{M}_0(C):\om=\pi(x,{\xi},s)\textrm{ , }
\bigl\|\frac{f_N(x)}{r_N^\ast(x)}-\Xi_\varphi\bigr\|\geq \e_0\}<\e_0
$.
Better yet, since $r_n^\ast\geq \lfloor n/n_0\rfloor\inf r_{n_0}^\ast$ tends to infinity uniformly and $f$ is bounded, we can take $N=N(\e_0,K^\ast)$ so large that  the set
$$
\Lambda(\e_0)\!:=\!\left\{\om\!=\!\pi(x,\xi,s)\!\in\! T^1M:\left\|\frac{f_N(x)+e_f}{r^\ast_N(x)+e_r}-\Xi_\vf\right\|\!<\!\e_0
\textrm{ for all }|e_r|,\|e_f\|\!<\!K^\ast\right\}
$$
satisfies
$
m_\vf[\Lambda(\e_0)\cap
\wt{M}_0(C)]>(1-\varepsilon_0)m_\vf[\wt{M}_0(C)]\mbox{ and }
m_\vf(\Lambda(\e_0)\cap
\wt{M}_0)>1-\varepsilon_0$.
This completes step 1.

\medskip
We know that $\L(\e_0)$ is large with respect to $m_\vf$. The next step is to study the size of the intersection of $\L(\e_0)$ with the horocycle of $\om_0$.  Recall from (\ref{lambda}) the definition of $A_T$ and $\l_T$, and define
$\l_T^C$ to be the  length measure on $A_T\cap\wt{M}_0(C)$, normalized s.t. $\l_T^C[\wt{M}_0]=1$:
\begin{eqnarray*}
\l_T^C&:=&\frac{1}{\int_0^T 1_{\wt{M}_0}(h^t(\om_0))dt}  \int_0^T
1_{\wt{M}_0(C)}(h^t\om_0)\d_{h^t(\om_0)}dt.
\end{eqnarray*}
Note that $\l_T$ is a probability measure, but $\l_T^C$ is not.

\medskip
\noindent
{\em Step 2.\/  If $\om_0$ is generic for $m_\vf$, then
there exists $\tau=\tau(\e_0)$ such that   for all $T\ge \tau$,  $\l_T(\wt{M}_0\setminus\Lambda(\e_0))\leq
\l_T^C(\wt{M}_0(C)\setminus\Lambda(\e_0))<2\e_0$.}

\medskip
\noindent
{\em Proof.\/}
Set $\L=\L(\e_0)$, $N=N(\e_0)$.
The discontinuities of the map
$\om\mapsto \frac{f_N(x(\om))}{r_N^\ast(x(\om))}$ are contained in the union of the  boundaries of all sets of the form
$
\pi\{(x,\xi,t):x\in [a_0,\ldots,a_{N-1}], \xi\in\Z^d, 0\leq t<r^\ast(x)\}$.
The Bowen--Series coding of \cite{Se1,Se2} has the property that these boundaries are geodesics arcs. Therefore, $m_\vf$ assigns to this set measure zero.

Thus for every $\e>0$,  we can construct sets $F_\e\subset U_\e\subset\L\cap \wt{M}_0(C)$ such that $F_\e$ is compact,  $U_\e$ is open, and $m_\vf[\L\cap\wt{M}_0(C)\setminus F_\e]<\e$.  Urysohn's lemma provides a continuous function  $0\leq \rho_\e(\cdot)\leq 1$ such that $\rho_\e=1$ on $F_\e$ and $\rho_\e=0$ outside $U_\e$ (whence outside $\L$). This function has compact support, and
$$
\int \rho_\e dm_\vf>m_\vf(\L\cap \wt{M}_0(C))-\e>1-(\e_0+\e).
$$

Next construct a continuous function with compact support $\theta_\e(\cdot)$ which approximates the indicator function of ${\wt{M}_0}$ from above in the sense that $0\leq \theta_\e\leq 1$, $\theta_\e=1$ on the closure of $\wt{M}_0$, and
$1<\int \theta_\e dm_\vf<1+\e$.

By construction, and since $\om_0$ is assumed to be $m_\vf$--generic,
$$
\l_T^C(\L)=\frac{\int_0^T 1_{\L\cap \wt{M}_0(C)}(h^t\om_0)dt}{\int_0^T 1_{\wt{M}_0}(h^t\om_0)dt}\geq \frac{\int_0^T \rho_\e(h^t\om_0)dt}{\int_0^T \theta_\e(h^t\om_0)dt}\xrightarrow[T\to\infty]{}\frac{\int\rho_\e dm_\vf }{\int\theta_\e dm_\vf }>\frac{1-(\e_0+\e)}{1+\e}.
$$
Thus there exists $\tau_\e$ such that for all $T>\tau_\e$, $\l_T^C(\L)>[1-(\e_0+\e)]/(1+\e)$. Choosing $\e$ sufficiently small,  we get that $\l_T^C(\L)>1-2\e_0$ for all $T>\tau_\e$ sufficiently large. Equivalently, $\l_T^C(\wt{M}_0(C)\setminus\L)<2\e_0$ for all $T>\tau_\e$. Note that $\tau_\e$ is a function of $\e_0$. This finishes the proof of the second step.

\medskip
The key lemma calls for the estimation of $\l_T(\wt{M}_0\cap \{r^\ast_{nN}<\ln T^\ast-K\}\setminus\L_n)$ for suitable choice of $K=K(\e_0)$, where here and throughout
$N=N(\e_0), \L=\L(\e_0)$, $\L_n=\L_n(N(\e_0),\e_0)$,  $T^\ast:=\ln(T/T_0)$ (with $T_0$ as above), and
$$
\{r^\ast_{nN}<T^\ast-K\}:=\{\om\in T^1M:r^\ast_{nN}(x(\om))<T^\ast-K \}.
$$

Our choice for $K$ is
$
K=K(\e_0):=10(N+L)\max r^\ast$.
We estimate the $\l_T$--measure of $\wt{M}_0\cap \{r^\ast_{nN}<T^\ast-K\}\setminus\L_n$ by partitioning $\wt{M}_0\cap \{r^\ast_{nN}<T^\ast-K\}\setminus\L_n(\e_0)$ into two pieces, which we then treat separately.
Define for this purpose
\begin{eqnarray*}
\mathcal W(\I)&:=&\bigcup W^{ss}_{\loc}(\om'),\textrm{ where the union is over all $\om'$ s.t. }W^{ss}_{\loc}(\om')\subseteq g^{T^\ast}[A_T],\\
\mathcal W(\II)&:=&\bigcup W^{ss}_{\loc}(\om'),\textrm{ where the union is over all $\om'$ s.t. }W^{ss}_{\loc}(\om')\not\subseteq g^{T^\ast}[A_T],\\
&&\textrm{but }\W^{ss}_{\loc}(\om')\cap g^{T^\ast}[A_T]\textrm{ has positive length}.
\end{eqnarray*}
Note that $\mathcal W(\I)$ and $\mathcal W(\II)$ are unions of local stable manifolds. Let $\mathcal N(\I), \mathcal N(\II)$ be the minimal number of local stable manifolds needed to cover these sets, up to sets of length zero.
Next define
\begin{eqnarray*}
A_T(\I)&:=& A_T\cap\wt{M}_0\cap \{r^\ast_{nN}<T^\ast-K\}\cap g^{-T^\ast}\left[\mathcal W(\I)\right],\\
A_T(\II)&:=& A_T\cap \wt{M}_0\cap \{r^\ast_{nN}<T^\ast-K\}\cap g^{-T^\ast}\left[\mathcal W(\II)\right].
\end{eqnarray*}
Clearly $A_T\setminus\L_n=[A_T(\I)\cup A_T(\II)]\setminus\L_n$.
The plan is to fix $n$, and estimate $\l_T[A_T(\I)\setminus\L_n]$ and $\l_T[A_T(\II)\setminus\L_n]$.

 \medskip
 \noindent
{\em Step 3.\/ $\sup_n\l_T[A_T(\I)\setminus\L_n]\leq 2C_r\e_0$ for all $T>\tau(\e_0)$. }

\medskip
\noindent{\em Proof.\/}
 We use the holonomy $\kappa_{n,N}$ which exchanges the  first and $n$-th blocks of $N$ symbols in $x$, see lemma \ref{properties_of_kappa_nN} above. The idea is to show that
\begin{equation}\label{thanks}
\kappa_{n,N}(A_T(\I)\setminus \L_n)\subset A_T\cap(\wt{M}_0(C)\setminus\L),
\end{equation}
which implies by
Lemma \ref{properties_of_kappa_nN} and the definition of $C_r$  that
\begin{eqnarray*}
\l_T[A_T(\I)\setminus \L_n]&=&\frac{\ell[A_T(\I)\setminus\L_n]}{\ell[A_T\cap\wt{M}_0]}\leq C_r\frac{\ell\circ\kappa_{n,N}[A_T(\I)\setminus\L_n]}{\ell[A_T\cap\wt{M}_0]}\\
&\leq &C_r\frac{\ell[A_T\cap\wt{M}_0(C)\setminus\L]}{\ell[A_T\cap\wt{M}_0]}\ \ \ \ \because(\ref{thanks})\\
&=&C_r\l_T^C(\wt{M}_0(C)\setminus\L)<2C_r\e_0\ \ \ (\textrm{step 2}),
\end{eqnarray*}
which proves the step.

We prove (\ref{thanks}). Suppose $\om=\pi(x,\xi,s+h(x))\in A_T(\I)\setminus\L_n$, and let $\om':=\kappa_{n,N}(\om)$.
\begin{enumerate}
\item $\om'\in A_T$:  The choice of $K$ is such that  $\om\in \{r_{nN}^\ast<T^\ast-K\}$ forces
$$
\hspace{1cm}r^\ast_{(n+1)N+L}(\om)\leq r^\ast_{nN}(\om)+(N+L)\max r^\ast<T^\ast.
$$
This implies that   $x(g^{T^\ast}\om)_0^\infty=x_p^\infty$ for $p>(n+1)N+L$.

 There exists some $m>0$ such that  $(\kappa_{n,N}^\ast(x))_{m}^\infty=x_{(n+1)N}^\infty$. By lemma \ref{properties_of_kappa_nN}, $x(\om')=\s^l(\kappa_{n,N}^\ast(x))$ with $|l|<L$, so  there must exist some $k\geq 0$ such that $x(\om')_k^\infty=x_{(n+1)N+L}^\infty$, whence  $\exists q>0$ such that $x(\om')_q^\infty=x_p^\infty=x(g^{T^\ast}\om)_0^\infty$. It is not difficult to see that this entails the existence of  $T^\#>0$ s.t.
$$
g^{T^\#}(\om')\in W^{ss}_{\loc}(g^{T^\ast}\om)\subset\mathcal W(\I)\subset g^{T^\ast}[A_T].
$$
Thus  $\om'\in g^{T^\ast-T^\#}(A_T)$. But image of $\kappa_{n,N}$ is always in $\Hor(\om_0)$, so we must have $T^\#=T^\ast$, and
$\om'\in A_T$.
\item $\om'\in \wt{M}_0(C)$: See lemma
\ref{properties_of_kappa_nN} part 3, and the choice of $E$ and $C$.
\item $\om'\not\in\L$: By construction, the first $N$-block of $x(\om')$ differs from the $n$-th $N$--block of $x(\om)$ by at most $2|l|<2L$ symbols. This means that
\begin{eqnarray*}
\hspace{1.5cm}\|f_N(x(\om'))-f_N(\s^{Nn}x(\om))\|&\leq &2L\max\|f\|<K^\ast\\
|r_N^\ast(x(\om'))-r_N^\ast(\s^{Nn}x(\om))|&\leq &|r_N(x(\om'))-r_N(x(\om))|+4\max|h|\\
&\leq &2L\max|r|+\var(r)+4\max|h|<K^\ast.
\end{eqnarray*}
Had $\om'$ been in $\L$, then it would have followed from these estimates that $\|\frac{f_N(\s^{Nn}x(\om))}{r^\ast_N(\s^{Nn}x(\om))}-\Xi_\vf\|<\e_0$, contrary to the assumption that $\om\not\in\L_n$.
\end{enumerate}
This finishes the proof of
(\ref{thanks}), and completes the step.

\medskip
\noindent
{\em Step 4.\/  $\sup_n\l_T[A_T(\II)\setminus\Lambda_n]\leq \frac{4C_r^2 d_{\max}}{\ell_{\min}}\e_0$ for all $T>\tau(\e_0)$.}

\medskip
\noindent
{\em Proof.\/}
Fix  $W^{ss}_{\loc}(g^{T^\ast}\om_1)\subset\mathcal W(\II)$ and
$W^{ss}_{\loc}(g^{T^\ast}\om_2)\subset \mathcal W(\I)$. We use $\vartheta_{T^\ast,\om_1,\om_2}$ to map  $g^{-T^\ast}W^{ss}_{\loc}(g^{T^\ast}\om_1)$ into  $g^{-T^\ast}W^{ss}_{\loc}(g^{T^\ast}\om_2)\in g^{-T^\ast}[\mathcal W(\I)]\subset A_T$, and then apply $\kappa_{N,n}$ to claim that
\begin{multline}\label{giving}
\kappa_{n,N}\circ\vartheta_{T^\ast,\om_1,\om_2}
\left(g^{-T^\ast}W^{ss}_{\loc}(g^{T^\ast}\om_1)\cap\wt{M}_0
\cap\{r_{nN}^\ast<T^\ast-K\}\setminus\L_n\right)\\\subset A_T\cap\wt{M}_0(C)\setminus\L.
\end{multline}

Suppose this were proved. Using the bounds on the  Radon-Nikodym derivatives of $\kappa_{n,N}$ and $\vartheta_{T^\ast,\om_1,\om_2}$, we can then deduce that
$$
\l_T\left(\!g^{-T^\ast}W^{ss}_{\loc}(g^{T^\ast}\om_1)\cap
\wt{M}_0\cap\{r_{nN}^\ast<T^\ast\!-\!K\}\setminus\L_n\!\right)\leq C_r^2 \l_T^C(\wt{M}_0(C)\setminus\L)<2C_r^2\e_0.
$$
Fixing $W^{ss}_{\loc}(\om_2)$, and  summing over all  $W^{ss}_{\loc}(g^{T^\ast}\om_1)$ needed to cover $\mathcal W(\II)$ up to sets of length zero,  we get
$
\l_T\left(A_T(\II)\setminus\L_n\right)\leq 2C_r^2\mathcal N(\II)\e_0$.

We claim that
$
\mathcal N(\II)\leq 2d_{\max}/\ell_{\min}$. To see this note that
every local stable manifold in $\mathcal W(\II)$ has length at least $\min\psi$, and is located at most $d_{\max}$ units of distance away from the endpoints of $g^{T^\ast}(A_T)$ (otherwise it would be contained in $g^{T^\ast}(A_T)$, which it is not). Since local stable manifolds are either equal or disjoint up to sets of length zero, this means that $\mathcal N(\II)\leq 2d_{\max}/\ell_{\min}$.

The step follows from this.
Thus it is enough to prove (\ref{giving}).

 The argument is similar to the one we used in step 3. Fix $\om\in g^{-T^\ast}W^{ss}_{\loc}(g^{T^\ast}\om_1)\cap\wt{M}_0
 \cap\{r_{nN}^\ast<T^\ast-K\}\setminus\L_n$ and set $\om':=(\kappa_{n,N}\circ\vartheta_{T^\ast,\om_1,\om_2})(\om)$.
\begin{enumerate}
\item $\om'\in A_T$: Set $\om'':=\vartheta_{T^\ast,\om_1,\om_2}(\om)$. Then
$\om''\in g^{-T^\ast}W^{ss}_{\loc}(g^{T^\ast}\om_2)$, and  $\exists k>0$, $l\in[-L,L]$ such that $|r_k^\ast(x(\om))-T^\ast|<5(\max r^\ast+\max |h|)$,
$x(\om'')_{k+M_{br}+l}^\infty=x(\om_2)_k^\infty$, and
$x(\om'')_{l}^{k+l}=x(\om)_0^k$ (lemma \ref{properties_of_vartheta}). This $k$ is larger than $(n+1)N+2L$, because by choice of $K$,
\begin{eqnarray*}
\hspace{1.5cm}r^\ast_{(n+1)N+2L}(x(\om))&\leq& r_{nN}^\ast(x(\om))+(N+2L)\max r^\ast\\
&<&T^\ast-5(\max r^\ast+\max |h|),
\end{eqnarray*}
and  $r_k^\ast(x(\om))>T^\ast-5(\max r^\ast+\max |h|)$.
Thus the tail $x(\om_2)_k^\infty$ survives $\kappa_{n,N}$, which means that $x(\om_2)_k^\infty$ appears as some tail of $x(\om')$.
This means that there exists $T^\#>0$ such that $g^{T^\#}(\om')\in W^{ss}_{\loc}(g^{T^\ast}\om_2)\subset g^{T^\ast}(A_T)$. As before $T^\#$ must be equal to $T^\ast$, whence $\om'\in A_T$.
\item $\om'\in \wt{M}_0$: $\vartheta_{T^\ast,\om_1,\om_2}$ moves points at most $E''$ units of distance. $\kappa_{n,N}$  moves points at most $E'$ units of distance. In total we move at most $2E$ units of distance from $\wt{M}_0$, which still leaves us well inside $\wt{M}_0(C)$.
\item $\om'\not\in\L$:  By construction, the first $N$--block of $x(\om')$ differs from the $n$-th $N$-block of $x$ by at most $4L$ coordinates (the edge effects of two shifts by at most $L$ units). Similar considerations to those used above show that had $\om'$ belonged to $\L$, then $\om$ would have had to belong to $\L_n$, a contradiction.
\end{enumerate}
(\ref{giving}) is proved and step 4 is done.
\medskip

Steps 3 and 4 imply the key lemma,  with
$C_0:=2C_r+4C_r^2 d_{\max}/\ell_{\min}$, and  $K(\e_0),T_0, N(\e_0), \tau(\e_0)$ as above.
 \hfill$\Box$


\section{Proof that a vector with asymptotic cycle in $\textrm{\em int}(\mathfrak C)$ is generic}
The purpose of this section is to prove
\begin{prop}\label{directprop}
For every $\e>0$ there is a  compact neighborhood $K_\vf(\e)\owns\Xi_\vf$ in $\R^d$ s.t. for all non-negative, non-identically zero continuous functions $f,g$, if  $T$ is large enough then
\begin{equation}\label{gribki}
\frac{\xi_{\ln T}(\om)}{\ln T}\in K_\vf(\e)\Longrightarrow e^{-\e}\frac{\int f dm_\vf}{\int g dm_\vf}\leq \frac{\int_0^T f(h^t\om)dt}{\int_0^Tg(h^t\om)dt}\leq e^{\e}\frac{\int f dm_\vf}{\int g dm_\vf}.
\end{equation}
\end{prop}

Thus  every vector with asymptotic cycle $\Xi_\vf$ is generic for $m_\vf$. Indeed we have the stronger statement that if there is a sequence $T_n\to\infty$ such that
$
\xi_{T_n}(\om)/T_n\to\Xi_\vf,
$
then ${\int_0^{e^{T_n}} f(h^t\om)dt}/{\int_0^{e^{T_n}}g(h^t\om)dt}\to{\int f dm_\vf}/{\int g dm_\vf}$ for all $f,g$ as above.

Another way of looking at this is to say that if $\xi_T(\om)/T$ has more than one accumulation point in ${int}(\mathfrak C)$ then ${\int_0^{T} f(h^t\om)dt}/{\int_0^{T}g(h^t\om)dt}$ oscillates without converging, and therefore cannot be  generic for any measure. This agrees with proposition \ref{converseprop}, but it does not prove it, because of the lack of information  on $\om$'s such that $\xi_T(\om)/T\to\del\mathfrak C$.

In what follows, we fix $\e>0$ and look for $K_\vf(\e)$ as above.

\subsection*{Modification of the Coding}
Fix  some small $\e^\ast=\e^\ast(\e)>0$, to be determined later.
Recall the symbolic coding of section \ref{symbsection}, in particular the definitions of $\psi$ and $\Psi_\vf$.
 Let $d_{\max}$ denote  the maximal diameter of a symbolic local stable manifold, measured in the intrinsic metric of the horocycle which contains it.

  We claim that we can  change the coding to ensure that  the new roof function and symbolic strong stable manifolds satisfy
\begin{eqnarray*}
\max r^\ast_{\modi}&<&\e^\ast,\\
\max |h_{\modi}|&<&\e^\ast\\
(d_{\max})_{\modi}&<&\e^\ast,\\
\max\psi_{\modi}&<&\e^\ast,
\end{eqnarray*}
$$
\diam\bigl(\pi_{\modi}\{(x,\xi,s):x_0=a_0, \xi=\xi_0, 0\leq s<r_{\modi}^\ast(x)\}\bigr)< \e^\ast\textrm{ for all }a_0, \xi_0.
$$
Moreover, we claim that this modification can be done in such a way that
$$
\frac{\max\psi_{\modi}}{\min\psi_{\modi}}, \frac{\max(\Psi_\vf)_{\modi}}{\min(\Psi_\vf)_{\modi}}<C_\vf
$$
where $C_\vf$ does not depend on $\e^\ast$ or $\e$.

Here is how to do this.

The coding in lemma \ref{bowen} is based on finding a Poincar\'e section for the \emph{geodesic} flow, with a Markov section map $T$.
Take a  power $\a_L:=\bigvee_{i=-L}^L T^{-i}\a$ of the Markov partition $\a$ of this section,
with  $L$ so large that
$$
V(L):=\sum_{k=L}^\infty \sup\{|r(x)-r(y)|:x_{-k}^k=y_{-k}^k\}<\tfrac{1}{2}\e^\ast.
$$
 Such $L$ exist because $r$ is H\"older continuous.

Increase the Poincar\'e section by adding to it the sets
$$
\textrm{$g^{k\e^\ast/2}(A)$, for  }A\in\a_L\textrm{ and all }k=1,\ldots,\lfloor 2\min_A r/\e^\ast\rfloor.
$$
This is again a Poincar\'e section for the geodesic flow. It can be used to code the geodesic flow as a special flow over a (new) subshift of finite type with roof function $r^\ast_{\modi}$ such that $\max r^\ast_{\modi}\leq\e^\ast/2<\e^\ast$ and $\var(r^\ast_{\modi})\leq V(L)<\e^\ast$.

Recall that $h_{\modi}$ is (any) H\"older continuous function such that $r_{\modi}:=r^\ast_{\modi}+h_{\modi}\circ\s_{\modi}-h_{\modi}$ only depends on the non-negative coordinates (w.r.t  $\Sigma_{\modi}$). There are  explicit constructions of such functions which lead to functions $h_{\modi}$ such that  $|h_{\modi}|\leq \var(r^\ast_{\modi})$, see for example the proof of  lemma 1.6 in \cite{Bo2}. Since $\var(r^\ast)<\e^\ast$, this gives $h_{\modi}$ with $|h_{\modi}|<\e^\ast$.

We claim that if $L$ is large enough, then the diameter and length of all symbolic manifolds in the new coding are less than $\e^\ast/2$, whence $(d_{\max})_{\modi}, \max(\psi_{\modi})<\e^\ast$.

To see this observe that for every $x$, the sets
$
\pi\{(y,\xi,s+h(y)):y_{-L}^L=x_{-L}^L\}$
decrease to $\{\pi(x,\xi,s+h(x))\}$ as $L\to\infty$. Thus for every $x$ there exists $L(x)$ such that the length and diameter of
$$
\pi\left(\{(y,\xi,s): y_{-L(x)}^{L(x)}=x_{-L(x)}^{L(x)}, \xi=\xi_0, 0<s<\e^\ast/2\}\right)
$$
are less than $\e^\ast$. (This $L(x)$ does not depend on $\xi_0$, because the deck transformations are isometries.)
A compactness argument shows that $L(x)$ can be chosen to be uniform in $x$.

We check that the ratio of the maxima and minima of $\psi_{\modi}$ and $(\Psi_{\vf})_{\modi}$ remains  bounded by a constant independent of $L$, $\e^\ast$ and $\e$.

First note that our modification of the coding does not change the value of $P(u_\vf)$, because as mentioned above the diffeomorphism $P(\cdot)$ does not depend on the coding (it can be described by a `section--free' formula, see \cite{BL}).

Unfortunately $(\Psi_\vf)_{\modi}:\Sigma^+_{\modi}\to\R$ is affected by the coding,   because it is defined to be the normalized eigenfunction of
$L_{-P(u_\vf),u_\vf}^{\modi}:C(\Sigma^+_{\modi})\to C(\Sigma^+_{\modi})$
and $r_{\modi},f_{\modi}$ and $\Sigma^+_{\modi}$ are different from $r,f$ and $\Sigma^+$.

The first observation is that any change which is solely due to refining the Markov partition
 does not affect the minimum of maximum of $\Psi_\vf$,
because the subshift of finite type it generates is conjugate to the original subshift.
It is therefore enough to focus on changes in the roof function $r$ and
the Poincar\'e section.

The modified  Poincar\'e section can be coded using a subshift of finite type $\Sigma^+_{\modi}$ with two  types of states: the old set of states $S$ of $\Sigma^+$,  and a new set of states  $S'$ each of which has a unique predecessor (i.e. a state which can lead to it).
It is easy to see that $\max r_{\modi}\leq\max r$, and  $\max\|f_{\modi}\|\leq \max\|f\|$ (in fact $f_{\modi}=0$ except perhaps on states which lead to $S$).

One can recover $\Sigma^+$ from $\Sigma^+_{\modi}$ by inducing: If one induces the shift $\s_{\modi}:\Sigma^+_{\modi}\to \Sigma^+_{\modi}$ on the union of the states in  $S$, then  one gets a dynamical system  isomorphic to $\s:\Sigma^+\to\Sigma^+$.
The same induction procedure can be used to recover $L_{-P(u_\vf),u_\vf}$ from $L_{-P(u_\vf),u_\vf}^{\modi}$:
If $\tau$ is the first return time to the states in $S$,
then $L_{-P(u_\vf),u_\vf}$ is conjugate to the operator  $F\mapsto [(L_{-P(u_\vf),u_\vf}^{\modi})^{\tau(x)}F](x)$ on $C(\Sigma^+_{\modi})$.
Standard results on inducing Ruelle operators imply that this conjugacy maps $(\Psi_\vf)_{\modi}|_{\bigcup S}$ to $\Psi_\vf$.
It follows that
$$
\max(\Psi_\vf)_{\modi}|_{S}=\max\Psi_\vf\textrm{ and }
\min(\Psi_\vf)_{\modi}|_{\bigcup S}=\min\Psi_\vf.
$$

Consider now points  $x\in [a]\subset\bigcup S'$. Let $k$ be the minimal natural number  such that $x=\s^k(y)$ and $y\in \bigcup S$. Since  every state in $S'$ has exactly one predecessor,  and since  $\Psi_\vf= L_{-P(u_\vf),u_\vf} \Psi_\vf$,
$$
\Psi_\vf(x)=e^{-P(u_\vf)(r_{\modi})_k(y)+\<u_\vf,(f_{\modi})_k(y)\>}\Psi_\vf(y).
$$
It is easy to verify that
 $|(r_{\modi})_k(y)|\leq |r^\ast(y)|+2\max |h|$ and that $(f_{\modi})_k(y)$ is either equal to zero or to $f(y)$. We conclude that
\begin{eqnarray*}
\Psi_\vf(x)&\leq & e^{|P(u_\vf)|(\max r^\ast+2\max|h|)+\|u_\vf\|\max\|f\|}\max(\Psi_\vf)_{\modi}|_{\bigcup S}  \\
\Psi_\vf(x)&\geq & e^{-|P(u_\vf)|(\max r^\ast+2\max|h|)-\|u_\vf\|\max\|f\|}\min(\Psi_\vf)_{\modi}|_{\bigcup S}.
\end{eqnarray*}

It follows that
$$
\frac{\max(\Psi_\vf)_{\modi}}{\min(\Psi_\vf)_{\modi}}\leq \const \x\frac{\max\Psi_\vf}{\min\Psi_\vf}
$$
where the constant is independent of $\e^\ast$, $\e$.

The bound on $\max \psi_{\modi}/\min\psi_{\modi}$ can be obtained in exactly the same way. Indeed, $\psi_{\modi}=(\Psi_0)_{\modi}$ where $0$ is the trivial homomorphism, see proposition 4.5 in \cite{BL}.

\medskip
This shows that the new coding is as required.
Henceforth we work with this coding, and drop the decorations by `new'.

\subsection*{Basic Sets}
Working with $\e, \e^\ast$ and  the symbolic coding of the previous section, we consider the following sets, which we name {\em $\e^\ast$--basic sets}, (or just `basic sets'):
\begin{equation}
\label{E}
E=\pi\{(x,\xi,s+h(x)):x\in[\un{a}]:=[\dot{a}_0,\ldots,a_{n-1}], \a\leq s<\b\}\subset\Sigma\x\Z^d\x\R,
\end{equation}
where $0\leq \a\leq\b\leq  \inf_{[\un{a}]}r^\ast$.
We call $\b-\a$ the {\em width} of $E$.

Note that $\diam(E)<\e^\ast$, and that
$$
m_\vf(E)=\const e^{\<u_\vf,\xi\>}\int_\a^\b e^{-P(u_\vf) s}ds\int_{[\un{a}]}\psi d\nu_{\vf}
$$
with the constant independent of $\e^\ast, \e$.
\begin{lem}\label{basiclemma}
If $\e^\ast$ is sufficiently small, then there is a
 compact neighborhood $K_\vf\owns\Xi_\vf$ in $\R^d$ s.t. for  any two $\e^\ast$--basic sets $E_1, E_2$, if $T$ is large enough
$$
 \frac{\xi_{\ln T}(\om)}{\ln T}\in K_\vf\Longrightarrow e^{-\e}\frac{m_\vf(E_1)}{m_\vf(E_2)}\leq \frac{\int_0^T 1_{E_1}(h^t\om)dt}{\int_0^T 1_{E_2}(h^t\om)dt}\leq e^{\e}\frac{m_\vf(E_1)}{m_\vf(E_2)}.
$$
\end{lem}
\begin{proof}
Lemma \ref{basiclemma} is proved  by finding  the asymptotic behavior of
$
I_T(\om,E):=\int_0^T 1_E(h^t\om)dt$
as $T\to\infty$.
Such asymptotics can be derived  for {\em almost every} $\om$ when $\vf\equiv 0$ from the main lemma of \cite{LS}. Here we describe the  necessary modifications to cover $\vf\not\equiv 0$, and to replace the quantification `$m_\vf$--almost everywhere' by  a condition on $\xi(g^{\ln T}\om)/\ln T$.

In what follows we assume without loss of generality that $\om\in \wt{M}_0$, and denote  $A_T(\om):=\{h^t(\om):0\leq t\leq T\}$ and $T^\ast:=\ln T$.

\medskip
\noindent
{\em Step 1.\/}  {\em
 For all $\om$ and $T>1$,
 $\exists N^+, N^-\in\N$ and $\exists\om_i^\ast\in
 g^{T^\ast}(A_T(\om))=A_{1}(g^{T^\ast}\om)$ $(i=0,\ldots,N^+)$ s.t.  if $
J_{T^\ast}(\om^\ast_i,E):=\ell[E\cap g^{-T^\ast}W^{ss}_{\loc}(\om^\ast_i)]$, then
 $$
\sum_{i=0}^{N^-}J_{T^\ast}(\om_i^\ast,E)\leq I_T(\om,E) \leq \sum_{i=0}^{N^+} J_{T^\ast}(\om_i^\ast,E) \textrm{ and } 0<\frac{N^+-N^-}{N^-}<\frac{4C_\vf\e^\ast}{1-2\e^\ast}.$$
}

\medskip
\noindent
{\em Proof.\/} This is the same as in step 1 in lemma 1 of \cite{LS}, except for the fact that we work in constant negative curvature, and this  allows us to argue for all $\om$, and not just almost all $\om$. We explain why.

We view $I_T(\om,E)$ as an integral on the horocyclic arc
$
A_T(\om)$
with respect to its hyperbolic length measure $\ell=\ell_\om$:
$$
I_T(\om,E)=\int_0^T 1_E(h^t\om)\/dt=\ell_\om[E\cap A_T(\om)]=\ell_\om[E\cap g^{-T^\ast}A_1(g^{T^\ast}\om)],
$$
where the last equality is because of the  commutation relation between the geodesic and horocycle flows.

Let $W^{ss}_{\loc}(\om_i^\ast)$ $(i=1,\ldots,N^-)$ be the symbolic local stable manifolds  contained in $A_{1}(g^{T^\ast}\om)$. Add to the list the $W^{ss}_{\loc}(\om_i^\ast)$ $(i=N^-+1,\ldots,N^+)$ which intersect $A_{1}(g^{T^\ast}(\om))$ with positive measure without being contained in it.

In constant negative curvature,  for {\em every} $\om$, any two symbolic local stable manifolds are either equal or disjoint up to sets of length zero (lemma \ref{bowen}, part 2). Therefore,   $I_T(\om,E)$ can be sandwiched between $\sum_{i=0}^{N^\pm}J_{T^\ast}(\om_i^\ast,E)$ as above.

The $N^+-N^-$ symbolic local stable manifolds which intersect $A_{1}(g^T\om)$ without being contained in it must be contained in a $d_{\max}$--neighborhood of the endpoints of $A_{1}(g^T\om)$, thus
$N^+-N^-<4d_{\max}/\min\psi$ and
$N^->(1-2d_{\max})/\max\psi$. Since $\max\psi/\min\psi<C_\vf$,
$
0<\frac{N^+-N^-}{N^-}<\frac{4C_\vf d_{\max}}{1-2d_{\max}}<\frac{4C_\vf\e^\ast}{1-2\e^\ast}.
$

\medskip
\noindent
{\em Step 2.\/} {\em Suppose $\om, \om_i^\ast$ have symbolic coordinates $(x,0,t+h(x))$,
$(x_i^\ast,\xi_i^\ast,t_i^\ast+h(x_i^\ast))$, and assume $T>e^{\e^\ast}$.
Set  $T_i^\#:=T^\ast-t_i^\ast$. Then for every $\e^\ast$--basic set $E$,
\begin{equation}\label{Jest}
J_{T^\ast}(\om_i^\ast,E)={e}^{\pm |\b-\a|}\sum_{n=0}^\infty\sum_{\s^n(y)=(x_i^\ast)_0^\infty}\!\!\!1_{[\a,\b]}(r_n(y)-T_i^\#)\d_{\xi_i^\ast, \xi+f_n(y)}1_{[\un{a}]}(y)\psi(y),
\end{equation}
where the  sum ranges over the {\em one}--sided subshift of finite type $\Sigma^+$.
}

\medskip
\noindent
{\em Proof.\/} See step 2 in lemma 1 of \cite{LS}.

\medskip
\noindent
We note for future reference that  $|T_i^\#-T^\ast|=|t_i^\ast|<\max r^\ast+\max|h|<2\e^\ast$, and that $|\b-\a|\leq \max r^\ast<\e^\ast$.

\medskip
\noindent
{\em Step 3  \cite{BL}}\/.
 {\em For every $\e^\ast$--basic set $E$, there exists  a compact
 neighborhood $K^0_\vf$ of $\Xi_\vf$ and  $T_0>0$  such that for all $T\ge T_0$,
$$
 \frac{\xi_i^\ast}{T^\ast}\in K_\vf^0\Rightarrow J_{T^\ast}(\om_i^\ast,E)=\const e^{\pm C_\vf\e^\ast}m_\vf(E)e^{\pm\e}
\frac{e^{T_i^\# H(\frac{\xi^\ast_i}{T_i^\#})}}{(\ln T)^{d/2}}\Psi_\vf(x_i^\ast),
$$
where $H(\cdot)$ is the minus the Legendre transform of   $P(\cdot)$, and $H(\cdot)$,  the constant  in front of the expression, $C_\vf$, and $K_\vf$ do not depend on $\e,\e^\ast, E$ or $\om$.
}

\medskip
\noindent
{\em Proof.\/}  This is done by estimating the sum (\ref{Jest})  using an elaboration of Lalley's method \cite{L}, as in Babillot \& Ledrappier \cite{BL1} or \cite{BL}.\footnote{For a detailed account of the calculation in the particular case $\vf\equiv 0$, see the appendix of \cite{LS}. The modifications needed to treat general $\vf$'s are routine.}

\medskip
\noindent
{\em Step 4.\/ There is
a function  $ \displaystyle F_\vf(\e,\e^\ast)\xrightarrow[\e,\e^\ast\to
  0^+]{} 1$,
s.t. for all $\e^\ast$--basic sets $E_1$, $E_2$, there is
 a compact  neighborhood $K_\vf\subset K_\vf^0$ of $\Xi_\vf$ and $T_1>0$ such that
for all
 $\om \in T^1 M$ and $T\ge T_1$
with $\frac{\xi_{T^\ast}(\om)}{T^\ast}\in K_\vf$,
$$
\frac{1}{F_\vf(\e,\e^\ast)}\frac{m_\vf(E_1)}{m_\vf(E_2)}\,\le
\frac{I_{T}(\om,E_1)}{I_T(\om, E_2)}\leq F_\vf(\e,\e^\ast)\frac{m_\vf(E_1)}{m_\vf(E_2)}\,.
$$
}

\medskip
\noindent
{\em Proof.}\/ It is enough to prove the inequality on the right.
By the previous steps,
\begin{equation}\label{inequality}
\frac{I_T(\om,E_1)}{I_T(\om,E_2)}\leq
e^{2C_\vf\e^\ast}e^{2\e}\frac{m_\vf(E_1)}{m_\vf(E_2)}
\left[1+\frac{\sum\limits_{i=N^-+1}^{N^+}e^{T_i^\#
      H(\frac{\xi_i^\ast}{T_i^\#})}\Psi_\vf(x_i^\ast)}{\sum\limits_{i=1}^{N^-}e^{T_i^\# H(\frac{\xi_i^\ast}{T_i^\#})}\Psi_\vf(x_i^\ast)}\right]
\end{equation}
Fix some small compact neighborhood $K_\vf\subset K_\vf^0$ of $\Xi_\vf$ (we shall see later how small), and assume $\frac{\xi_{T^\ast}(\om)}{T^\ast}\in K_\vf$.

By construction all the
$\om_i^\ast$  belong to a $d_{\max}$--neighborhood of  $A_{1}(g^{T^\ast}\om)$, a horocyclic arc of length $1$.  Thus their $\Z^d$--coordinates $\xi_i^\ast$ are within a bounded distance $D$ from each other.
 Since $|T_i^\#-T^\ast|<2\e^\ast$, for all $T$ large enough,
  $\|\frac{\xi_i^\ast}{T_i^\#}-\Xi_\vf\|<2\diam K_\vf$. Thus, if  $\diam K_\vf$ is small enough, then
$$
\hspace{-8cm}|T_i^\#H(\tfrac{\xi_i^\ast}{T_i^\#})-T^\ast H(\tfrac{\xi_{T^\ast}(\om)}{T^\ast})| \leq
$$
\begin{eqnarray*}
&\leq& |T_i^\#-T^\ast|\cdot |H(\tfrac{\xi_i^\ast}{T_i^\#})|+T^\ast|H(\tfrac{\xi_i^\ast}{T_i^\#})-H(\tfrac{\xi_{T^\ast}(\om)}{T^\ast})\bigr|\\
&\leq & 2\e^\ast[H(\Xi_\vf)+\e^\ast]+(\|(\nabla H)(\Xi_\vf)\|+\e^\ast)T^\ast\left\|\frac{\xi_i^\ast}{T_i^\#}-\frac{\xi_{T^\ast}(\om)}{T^\ast}\right\|.
\end{eqnarray*}
Now, for $T$ large enough so that $\|\frac{\xi_i^\ast}{T^\#_i}\|<2\diam K_\vf+\|\Xi_\vf\|$,  we have
\begin{eqnarray*}
T^\ast\left\|\frac{\xi_i^\ast}{T_i^\#}-\frac{\xi_{T^\ast}(\om)}{T^\ast}\right\|
&\leq& T^\ast
\left\|\frac{\xi_i^\ast}{T_i^\#}-\frac{\xi_i^\ast}{T^\ast}\right\|+
T^\ast \left\|\frac{\xi_i^\ast}{T^\ast}-\frac{\xi_{T^\ast}(\om)}{T^\ast}\right\|\\
&\leq
&T^\ast\left\|\frac{\xi_i^\ast}{T^\#_i}\right\|\left|1-\frac{T_i^\#}{T^\ast}\right|
+
D\\
&\leq &(2\diam K_\vf+\|\Xi_\vf\|)\cdot 2\e^\ast+D.
 \end{eqnarray*}
Consequently, there is a constant $k_\vf'$ which  only depends on $\vf$ such that for all $T$ large enough, if
$\frac{\xi(g^{T^\ast}\om)}{T^\ast}\in K_\vf$, then
$$
\exp\left(|T_i^\#H(\tfrac{\xi_i^\ast}{T_i^\#})-T^\ast H(\tfrac{\xi_{T^\ast}(\om)}{T^\ast})|\right)\leq \exp\left(k_\vf'\e^\ast+D\|\nabla H(\Xi_\vf)\|\right)=:k_\vf(\e^\ast).
$$

Thus the term in the brackets in (\ref{inequality}) can be estimated by
$$
1+k_\vf(\e^\ast)^2\frac{\max\Psi_\vf}{\min\Psi_\vf}\left(\frac{N^+-N^-}{N^-}\right)\leq 1+\frac{4 {\e^\ast}C_\vf^2 k_\vf(\e^\ast)^2}{1-2\e^\ast},
$$
and so $I_T(\om,E)/I_T(\om,E_0)\leq e^{2C_\vf\e^\ast}e^{2\e}\frac{m_\vf(E)}{m_\vf(E_0)}F_\vf(\e_0,\e^\ast)$ where
$$
F_\vf(\e,\e^\ast):=e^{2C_\vf\e^\ast+2\e}\left(1+\frac{4\e^\ast C_\vf k_\vf(\e^\ast)^2}{1-2\e^\ast}\right).
$$
The lower bound is obtained in the same way, and the step is proved.

\medskip
We can now prove the lemma. Fix $\wt{\e}$. Choose $\e,\e^\ast$ so small that
$e^{-\wt{\e}}<F(\e,\e^\ast)<e^{\wt{\e}}$. Then the lemma follows with $\wt{\e}$ instead of $\e$. \end{proof}

\subsection*{Proof of Proposition \ref{directprop}} Take $\e$ and $K_\vf(\e)$ as in the previous section, and fix $\e^\ast$ as in lemma \ref{basiclemma}. Call a function an $\e^\ast$--{\em step function} if it is a finite linear combination of indicators of $\e^\ast$--basic sets. Lemma \ref{basiclemma} clearly implies (\ref{gribki}) for any pair of $\e^\ast$--step functions with non-zero integrals.

 We also have (\ref{gribki}) for all  non-negative non-identically zero continuous functions $f,g$  with compact supports which do not intersect the section:
Such functions can be sandwiched between  $\e^\ast$--step functions  with almost the same integrals, and $\e^\ast$ and the width of the basic sets can be chosen arbitrarily small.

We claim that there is an open neighborhood $V$ of the section such that  for all  non-negative, non-identically zero, continuous functions $f,g$  supported inside $V$, if $T$ is large enough and $\xi_{\ln T}(\om)/\ln T\in int(K_\vf)$, then
$$
e^{-3\e}\frac{\int f dm_\vf}{\int g dm_\vf}\,\le
\frac{\int_0^T f(h^t\om)dt}{\int_0^T g(h^t\om)dt}\leq e^{3\e}\frac{\int f dm_\vf}{\int g dm_\vf}.
$$

Every point  $\wt{\om}$ in the section has a precompact open neighborhood $V_{\wt{\om}}$ and a constant $0<s(\wt{\om})<\e^\ast<\e$ such that the closure of $g^{s(\wt{\om})}[V_{\wt{\om}}]$ does not intersect the section. Suppose $f_1,f_2$ are non-negative, non-identically zero, continuous functions with  compact support  in  $V_{\wt{\om}}$;  then $\wt{f}_i:=f_i\circ g^{-s(\wt{\om})}$ are uniformly continuous functions with compact supports which do not intersect the section. Thus they satisfy (\ref{gribki}) for all $T$ large enough.

The commutation relation between the geodesic flow and the horocycle flow implies
$
\int_0^{T}{f}_i(h^t\om)dt
=e^{s(\wt{\om})}\int_0^{Te^{-s(\wt{\om})}}\wt{f}_i(h^{\tau}g^{-s(\wt{\om})}\om)d\tau
$.
If $\frac{\xi_{\ln T}(\om)}{\ln T}\in int(K_\vf)$, then for $T$ large enough  $\frac{\xi_{\ln T}(g^{s(\wt{\om})}\om)}{\ln T}\in int(K_\vf)$. For such $T$,
$ \int_0^{T}{f}_1(h^t\om)dt/\int_0^{T}{f}_2(h^t\om)dt= \int_0^{Te^{-s(\wt{\om})}}\wt{f}_1(h^\tau g^{-s(\wt{\om})}\om)d\tau/
\int_0^{Te^{-s(\wt{\om})}}\wt{f}_2(h^\tau g^{-s(\wt{\om})}\om)d\tau=e^{\pm \e}\int \wt{f}_1/\int \wt{f}_2$. The integral ratio is equal to $\int f_1/\int f_2$, because  $m_\vf\circ g^\tau=e^{-P(u_\vf)\tau}m_\vf$.

Thus (\ref{gribki}) holds for all $f_1,f_2$ supported inside $V_{\wt{\om}}$.

Now suppose $\supp f_1\subset V_{\wt{\om}_1}$ and $\supp f_2\subset V_{\wt{\om}_2}$ where $\wt{\om}_1\neq \wt{\om}_2$. Choose non-negative non-identically zero continuous $g_i$ with compact support in $V_{\wt{\om}_i}\setminus$section. Writing
$$
\frac{\int_0^T f_1\circ h^t dt}{\int_0^T f_2\circ h^t dt}=
\frac{\int_0^T f_1\circ h^t dt}{\int_0^T g_1\circ h^t dt}\cdot\frac{\int_0^T g_1\circ h^t dt}{\int_0^T g_2\circ h^t dt}\cdot\frac{\int_0^T g_2\circ h^t dt}{\int_0^T f_2\circ h^t dt},
$$
we see that for all $T$ large enough, if $\xi_{\ln T}(\om)/\ln T\in int(K_\vf)$ then
$$
e^{-3\e}\frac{\int f_1 dm_\vf}{\int f_2 dm_\vf}\,\le
\frac{\int_0^T f_1(h^t\om)dt}{\int_0^T f_2(h^t\om)dt}\leq e^{3\e}\frac{\int f_1 dm_\vf}{\int f_2 dm_\vf}.
$$

Let $V$  be the union of $V_{\wt{\om}}$ where $\wt{\om}$ ranges over the section. Every non-negative continuous function with compact support in $V$ is the sum of finitely many non-negative continuous functions supported inside $V_{\wt{\om}}$. Thus proposition \ref{directprop} holds for all $f,g$ non-negative, non-identically zero, continuous functions supported inside $V$.

\medskip
Now set $\mathcal U$ to be the family of all non-negative non-identically zero continuous functions with compact support, whose support does not intersect the section. Define $\mathcal V$ to be the family of all non-negative non-identically zero continuous functions supported inside $V$.

We saw that proposition \ref{directprop} holds for all pairs of functions in $\mathcal U$, and for all pairs of functions  in $\mathcal V$. Since $\mathcal U\cap\mathcal V\neq \emptyset$, proposition \ref{directprop} holds for all pairs of  functions in $\mathcal U\cup\mathcal V$. Since any non-negative continuous function of compact support is the sum of a function from $\mathcal U$ and a function from $\mathcal V$, proposition \ref{directprop} holds for all non-negative continuous functions of compact support.
\hfill $\Box$

\section{A Question}
The obvious question is what happens for other hyperbolic surfaces of infinite genus. Theorem \ref{mainthm} does not make sense for such surfaces, because the notion of asymptotic cycle is specific to $\Z^d$--covers.

It is  desirable to find another criterion which does make sense in general.

The following observation is perhaps a step in this direction. Let $\Delta$ denote the Laplace--Beltrami operator on the hyperbolic surface $M$. A positive eigenfunction $F:M\to\R$ is called {\em minimal}, if it defines an extremal ray in the cone of positive eigenfunctions with the same eigenvalue. The minimal positive eigenfunctions are known for $\Z^d$--covers \cite{CG}, \cite{LP}: they form a list $\{c F_\vf: c>0, \vf:\Z^d\to\R\textrm{ a homomorphism}\}$, where $F_\vf$ satisfies $F_\vf\circ D_{\xi}=e^{\vf(\xi)}F_\vf$ for all deck transformations $D_\xi$, $\xi\in\Z^d$. The similarity with the collection of Babillot--Ledrappier measures is not a coincidence, see \cite{LS2}.

Extend $F_\vf:M\to\R$ to $F_\vf:T^1 M\to\R$ by setting $F_\vf(\om)=F_\vf(\textrm{base point of $\om$})$. It is easy to see using the precompactness of $\wt{M}_0$ and the continuity of $F_\vf$ that for all homomorphisms $\vf$,
$$
\vf(\Xi(\om))=\lim\limits_{T\to\infty}\frac{1}{T}\ln F_\vf(g^T\om).
$$
Thus $\Xi(\om)$ is completely determined by the logarithmic growth of the minimal positive eigenfunctions of $\Delta$ along the forward geodesic ray of $\om$.

In particular, we get the following corollary of theorem \ref{mainthm}: Suppose $M$ is a $\Z^d$--cover of a compact hyperbolic surface, then $\om$ is generic for the horocycle flow with respect to the volume measure iff
$$
\lim\limits_{T\to\infty}\frac{1}{T}\ln F(g^T\om)=0
$$
for every positive minimal eigenfunction $F$ of the Laplace--Beltrami operator of $M$.

\subsection*{Question:} {\em  Does the above extend to other, perhaps all, hyperbolic surfaces with the Liouville property? }

\medskip
\noindent
We remind the reader that a hyperbolic surface is called {\em Liouville} if all its bounded harmonic functions are constant. We need to assume the Liouville property, because of Kaimanovich's theorem \cite{K}, which says that the volume measure is ergodic for the horocycle flow on a hyperbolic surface, iff this surface is Liouville.

\end{document}